\documentclass[a4paper,11pt,reqno]{amsart}

\usepackage[utf8]{inputenc}

\usepackage{xcolor}

\definecolor{verydarkblue}{rgb}{0,0,0.5}

\usepackage[
breaklinks,
colorlinks,
citecolor=verydarkblue,
linkcolor=verydarkblue,
urlcolor=verydarkblue,
hyperindex
]{hyperref}
\usepackage{fancyhdr}

\usepackage[
hscale=0.7,
vscale=0.75,
headheight=13pt,
marginpar=2cm,
centering,
]{geometry}

\usepackage{amsmath}
\usepackage{amsthm}
\usepackage{amssymb}
\usepackage{mathtools}
\usepackage{mathdots}
\usepackage{framed}
\usepackage[capitalize]{cleveref}
\usepackage{enumitem}
\usepackage{array}
\usepackage[all,cmtip]{xy}
\usepackage{tikz}
\usetikzlibrary{cd}
\newcommand{\comment}[1]{}
\newcommand{\tr}{\operatorname{tr}}
\newcommand{\ol}[1]{\overline{#1}}

\theoremstyle{plain}

\crefname{introtheorem}{Theorem}{Theorems}

\newtheorem*{theorem*}{Theorem}

\newtheorem{theorem}{Theorem}[section]
\newtheorem{proposition}[theorem]{Proposition}
\newtheorem{lemma}[theorem]{Lemma}
\newtheorem{corollary}[theorem]{Corollary}

\theoremstyle{definition}

\newtheorem{definition}[theorem]{Definition}
\newtheorem{notation}[theorem]{Notation}

\theoremstyle{remark}

\newtheorem{remark}[theorem]{Remark}
\newtheorem{example}[theorem]{Example}

\numberwithin{figure}{section}

\numberwithin{equation}{section}

\IfFileExists{./article-style.tex}{\input{article-style.tex}}{}

\usepackage{marginnote}

\def\N{{\mathbb N}}
\def\Z{{\mathbb Z}}

\newcommand{\PP}{\mathbb{P}}

\def\P{{\mathbb P}}

\def\cQ{\mathcal{Q}}

\def\cV{\mathcal{V}}

\def\.{\cdot}
\def\^{\widehat}
\def\~{\widetilde}

\def\({\left(}
\def\){\right)}

\def\*{{}^*}

\renewcommand{\and}{ \ \ \text{ and } \ \ }

\def\red{\mathrm{red}}

\renewcommand{\phi}{\varphi}

\DeclareMathOperator{\GL}  {GL}

\DeclareMathOperator{\Gr} {Gr}
\DeclareMathOperator{\SL}  {SL}

\DeclareMathOperator{\Spec} {Spec}

\DeclareMathOperator{\End} {End}

\DeclareMathOperator{\Sym} {Sym}

\DeclareMathOperator{\Cl} {Cl}
\DeclareMathOperator{\rad} {rad}

\DeclareMathOperator{\id} {id}

\DeclareMathOperator{\Id} {Id}

\DeclareMathOperator{\iso}{iso}

\def\isom{\simeq}

\newcommand{\Tim}[1]{}
\newcommand{\Chris}[1]{}
\newcommand{\Nafie}[1]{}
\newcommand{\Jan}[1]{}
\newcommand{\Rob}[1]{}

\DeclareMathOperator{\Spin}{Spin}
\DeclareMathOperator{\SO}{SO}
\renewcommand{\O}{\mathrm{O}}

\def\so{\mathfrak{so}}

\def\inv{\times} % subset of invertible elements
\def\Wedge{\bigwedge\nolimits}
\def\gl{\mathfrak{gl}}

\newcommand{\lspan}[1]{\langle #1 \rangle}

\begin{document}

\title{Topological Noetherianity of the infinite half-spin representations}

\author[Chiu]{Christopher Chiu}
\address{Department of Mathematics, KU Leuven, Celestijnenlaan 200b, box 02400, 3001 Heverlee, Belgium}
\email{christopherheng.chiu@kuleuven.be}

\author[Draisma]{Jan Draisma}
\address{Mathematical Institute, University of Bern, Sidlerstrasse 5, 3012 Bern, Switzerland}
\email{jan.draisma@unibe.ch}

\author[Eggermont]{Rob Eggermont}
\address{Department of Mathematics and Computer Science, Eindhoven University of Technology, P.O. Box 513, 5600MB, Eindhoven, the Netherlands}
\email{r.h.eggermont@tue.nl}

\author[Seynnaeve]{Tim Seynnaeve}
\address{Department of Computer Science, KU Leuven, Celestijnenlaan 200A, Room 02.183, 3001 Leuven, Belgium}
\email{tim.seynnaeve@kuleuven.be}

\author[Tairi]{Nafie Tairi}
\address{Mathematical Institute, University of Bern, Alpeneggstrasse 22, 3012 Bern, Switzerland}
\email{nafie.tairi@unibe.ch}

\maketitle

\begin{abstract}
     We prove that the infinite half-spin representations are topologically Noetherian with respect to the infinite spin group. As a consequence we obtain that half-spin varieties, which we introduce, are defined by the pullback of equations at a finite level. The main example for such varieties is the infinite isotropic Grassmannian in its spinor embedding, for which we explicitly determine its defining equations.
\end{abstract}

\section{Introduction} \label{s:intro}

\subsection{Purpose of this paper and main theorem} 

The purpose of this paper is to study certain varieties $X_n$ that live
in the half-spin representations of the even spin groups $\Spin(2n)$
with $n$ varying. In particular, we will show that these varieties are
defined, for all $n$, by pulling back the equations for a single $X_{n_0}$
along suitable contraction maps. The simplest instance of such a
variety is the Grassmannian of $n$-dimensional isotropic spaces in a
$2n$-dimensional orthogonal space. In this case, we use earlier work
\cite{ST21} by the last two authors to show that $n_0$ can be taken equal to $4$; see Theorem~\ref{cor:isotropic-cartan}.

But the {\em half-spin varieties} that we introduce go far beyond the
maximal isotropic Grassmannian. Indeed, this class of varieties is
preserved under linear operations such as joins and tangential
varieties, and under finite unions and arbitrary intersections.  Consequently, any variety obtained from several copies of the maximal isotropic Grassmannian by such operations is defined by equations of some degree
bounded independently of $n$. We stress, though, that these results
are of a purely topological/set-theoretic nature. It is not true,
for instance, that one gets the entire ideal of the maximal isotropic
Grassmannian of $n$-spaces in a $2n$-space by pulling back equations
for $X_4$ along the maps that we define.

Our main results about half-spin varieties are
Theorem~\ref{t:spin-var-top-noeth}, which establishes a descending chain
condition for these, and Corollary~\ref{cor:finite_to_finite_spinvariety},
which implies the results mentioned above. These results follow from
a companion result in infinite dimensions, which is a little easier to
state here. We will construct a direct limit $\Spin(V_\infty)$ of all
spin groups; here $V_\infty=\bigcup_n V_n$ is a countable-dimensional
vector space with basis $e_1,f_1,e_2,f_2,e_3,f_3,\ldots$ and a bilinear
form determined by $(e_i|e_j)=(f_i|f_j)=0$ and $(e_i|f_j)=\delta_{ij}$.
Furthermore, we will construct a direct limit $\Wedge_\infty^+ E_\infty$
of all even half-spin representations. This space has as basis all
formal infinite products
\[ e_{i_1} \wedge e_{i_2} \wedge e_{i_3} \wedge \cdots \]
where $\{i_1<i_2<\ldots\}$ is a cofinite subset of the positive
integers. The group $\Spin(V_\infty)$ acts naturally on this space,
and hence on its dual $(\Wedge_\infty^+ E_\infty)^*$, which we regard
as the spectrum of the symmetric algebra on $\Wedge_\infty^+ E_\infty$.
Our main theorem is as follows.

\begin{theorem} \label{thm:Main1}
The scheme $(\Wedge_\infty^+ E_\infty)^*$ is topologically
$\Spin(V_\infty)$-Noetherian. That is, every chain
\[ 
    X_1 \supseteq X_2 \supseteq X_3 \supseteq \ldots 
\]
of $\Spin(V_\infty)$-stable reduced closed subschemes stabilises.
\end{theorem}

\subsection{Relations to the literature}

Our work is primarily motivated by earlier work by the second and
third author on {\em Pl\"ucker varieties}, which live in exterior
powers $\Wedge^n K^{p+n}$ with both $p$ and $n$ varying. The results
in \cite{DE18} on Pl\"ucker varieties are analoguous to the results
we establish here for half-spin varieties, and the main result in
\cite{Nekrasov20} 
is an exact analogue of Theorem~\ref{thm:Main1} for
the dual infinite wedge, acted upon by the infinite general linear group.

On the one hand, we now have much better tools available to study these
kind of questions than we had at the time of \cite{DE18}---notably the
topological Noetherianity of polynomial functors \cite{Draisma19} and
their generalisation to algebraic representations \cite{Eggermont22}.
But on the other hand, spin representations are much more intricate
than polynomial functors, and a part of the current paper
will be devoted to establishing the precise relationship between the infinite half-spin representation and algebraic representations of the infinite general linear group, so as to use
those tools.

This paper fits in a general programme that asks for which sequences of
representations of increasing groups one can expect Noetherianity
results. This seems to be an extremely delicate question. Indeed,
while Theorem~\ref{thm:Main1} establishes Noetherianity of the dual
infinite {\em half-}spin representation, we do not know whether the
dual infinite spin
representation is $\Spin(V_\infty)$-Noetherian; see
Remark~\ref{rem:so-noetherianity}. Similarly, we do not
know whether a suitable inverse limit of exterior powers $\Wedge^n V_n$
is $\SO(V_\infty)$-Noetherian---and there are many more natural sequences
of representations for which we do not yet have satisfactory results.

In the context of secant varieties, we point out the work by Sam on
Veronese varieties: the $k$-th secant variety of the $d$-th Veronese
embedding of $\PP(K^n)$ is defined {\em ideal-theoretically} by finitely
many types of equations, independently of $n$---and in particular in
bounded degree \cite{Sam15b}. Furthermore, a similar statement holds for
the $p$-th syzygies for any fixed $p$ \cite{Sam16}. Similar results for ordinary
Grassmannians were established by Laudone in \cite{Laudone18}. It would be very
interesting to know whether their techniques apply to secant varieties
of the maximal isotropic Grassmannian in its spinor embedding. 
Our results here give a weaker set-theoretic statement, but for a more
general class of varieties.

After establishing Noetherianity, it would be natural to try and study
additional geometric properties of $\Spin(V_\infty)$-stable subvarieties of the
dual infinite half-spin representation. Perhaps there is a theory there
analogous to the theory of $\GL$-varietes \cite{Bik21,Bik22}. However,
we are currently quite far from any such deeper understanding!

\subsection{Organisation of this paper}

In \S\ref{s:spin-reps}, we recall the construction of the
(finite-dimensional) half-spin representations. We mostly do this in
a coordinate-free manner, only choosing---as one must---a maximal
isotropic subspace of an orthogonal space for the construction. But for
the construction of the infinite half-spin representation, we will need
explicit formulas, and these are derived in \S\ref{s:spin-reps}, as well.

In \S\ref{sec:IsoGras/InfiniteSpin}, we first describe the
embedding of the maximal isotropic Grassmannian in the projectivised
half-spin representation. Then, we define suitable contraction and
multiplication maps, which we show preserve the cones over these
isotropic Grassmannians. Finally, we use these maps to construct the
infinite-dimensional half-spin representations.

In \S\ref{s:noetherian}, we prove Theorem~\ref{thm:Main1} (see
Theorem~\ref{thm:Main}); and in \S\ref{s:spin-vars}, we state and prove
the main results about half-spin varieties discussed above. Finally, in \S\ref{sec:Cartan}
we prove the universality of the isotropic Grassmannian of $4$-spaces in an
$8$-dimensional space. We do so by relating the half-spin representations
via the Cartan map to the exterior power representations and using
results from \cite{ST21}.

\section{Finite spin representations and the spin group} 
\label{s:spin-reps}

In this section we collect some preliminaries on spin groups and their defining representations. Throughout we will assume that $K$ is an algebraically closed field of characteristic $0$. We follow \cite{manivel2009spinor} in our set-up; for more general references on spin groups and their representations see \cite{LM89,Pro07}. 

\subsection{The Clifford algebra}
\label{ssec:Clifford}

Let $V$ be a finite-dimensional vector space over $K$ endowed with a quadratic form $q$. The \textit{Clifford algebra} $\Cl(V,q)$ of $V$ is the quotient of the tensor algebra $T(V) = \bigoplus_{d \geq 0 } V^{\otimes d}$ by the two-sided ideal generated by all elements
\begin{equation} \label{eq:cliff-rel}
    v \otimes v - q(v) \cdot 1, \: v \in V.
\end{equation}
This is also the two-sided ideal generated by
\begin{equation} \label{eq:cliff-rel-2}
    v \otimes w + w \otimes v - 2 (v|w) \cdot 1, \: v,w \in V,
\end{equation}
where $(\cdot|\cdot)$ denotes the bilinear form associated to $q$ defined by $(v|w):=\frac{1}{2}(q(v+w)-q(v)-q(w))$.

The Clifford algebra is a functor from the category of vector spaces equipped with a quadratic form to the category of (unital) associative algebras. That is, any linear map $\varphi \colon (V,q) \to (V',q')$ with $q'(\varphi(v))=q(v)$ for all $v \in V$ induces a homomorphism of associative algebras $\Cl(\varphi) \colon \Cl(V,q) \to \Cl(V',q')$. If $\phi$ is an inclusion $V \subseteq V'$, then $\Cl(\varphi)$ is injective, and hence $\Cl(V,q)$ is a subalgebra of  $\Cl(V',q')$. 

The decomposition of $T(V)$ into the even part $T^+(V):=\bigoplus_{d \text{ even}} V^{\otimes d}$ and the odd part $T^-(V):=\bigoplus_{d \text{ odd}} V^{\otimes d}$ induces a decomposition $\Cl(V,q) = \Cl^+(V,q) \oplus \Cl^-(V,q)$, turning $\Cl(V,q)$ into a $\Z/2\Z$-graded associative algebra. Note that, via the commutator on $\Cl(V,q)$, the even Clifford algebra $\Cl^+(V,q)$ is a Lie subalgebra of $\Cl(V,q)$.

The anti-automorphism of $T(V)$ determined by $v_1 \otimes \cdots \otimes v_d \mapsto v_d \otimes \cdots \otimes v_1$ preserves the ideal in the definition of $\Cl(V,q)$ and therefore induces an anti-automorphism $x \mapsto x^*$ of $\Cl(V,q)$.

\subsection{The Grassmann algebra as a \texorpdfstring{$\Cl(V)$-module}{Cl(V)-module}} \label{ssec:Grassmann}

From now on, we will write $\Cl(V)$ for $\Cl(V,q)$ when $q$ is clear from the context. If $q=0$, then $\Cl(V) =\Wedge V$, the Grassmann algebra of $V$. If $E \subseteq V$ is an isotropic subspace, that is, a subspace for which $q|_E=0$, then this fact allows us to identify $\Wedge E$ with the subalgebra $\Cl(E)$ of $\Cl(V)$.

For general $q$, $\Cl(V)$ is not isomorphic as an algebra to $\Wedge V$, but $\Wedge V$ is naturally a $\Cl(V)$-module as follows. For $v \in V$ define $o(v):\Wedge V \to \Wedge V$ (the ``outer product'') as the linear map 
\[ 
    o(v)\omega:=v \wedge \omega 
\]
and $\iota(v):\Wedge V \to \Wedge V$ (the ``inner product'') as the linear map determined by 
\[ 
    \iota(v)w_1 \wedge \cdots \wedge w_k:=\sum_{i=1}^k (-1)^{i-1} (w_i\mid v) w_1 \wedge \cdots \wedge \widehat{w}_i \wedge \cdots \wedge w_k. 
\]
Here, and elsewhere in the paper, \; $\widehat{\cdot}$\; indicates a factor that is left out. Now $v \mapsto \iota(v) + o(v)$ extends to an algebra homomorphism $\Cl(V) \to \End(\Wedge V)$. To see this, it suffices to consider $v,w_1,\ldots,w_k \in V$ and verify 
\[ 
    (\iota(v)+o(v))^2 w_1 \wedge \cdots \wedge w_k = (v|v) w_1 \wedge \cdots \wedge w_k. 
\]
We write $a \bullet \omega$ for the outcome of $a \in \Cl(V)$ acting on $\omega \in \Wedge V$. Using induction on the degree of a product, the linear map $\Cl(V) \to \Wedge V, a \mapsto a \bullet 1$ is easily seen to be an isomorphism of vector spaces. In particular, $\Cl(V)$ has dimension $2^{\dim V}$. 

\subsection{Embedding \texorpdfstring{$\so(V)$}{so(V)} into the Clifford algebra} \label{ssec:embedding}

From now on, we assume that $q$ is nondegenerate and write
$\SO(V)=\SO(V,q)$ for the special orthogonal group of $q$. Its Lie
algebra $\so(V)$ consists of linear maps $V \to V$ that are
skew-symmetric with respect to $(\cdot|\cdot)$, that is, those $A \in
\End(V)$ such that  $(Av|w) = -(v|Aw)$ for all  $v, w \in V$. We have
a unique linear map $\psi: \Wedge^2 V \to \Cl^+(V)$ with $\psi(u
\wedge v)=uv-vu$, and $\psi$ is injective. A straightforward
computation shows that the image $L$ of $\psi$ is closed under the
commutator in $\Cl(V)$, hence a Lie subalgebra. We claim that $L$ is isomorphic to $\so(V)$. Indeed, for $u,v,w \in V$ we have  
\[
    [\psi(u \wedge v),w] = [[u,v],w] = 4(v|w) u - 4(u|w)v.
\]
We see, first, that $V \subseteq \Cl(V)$ is preserved under the adjoint
action of $L$; and second, that $L$ acts on $V$ via skew-symmetric
linear maps, so that $L$ maps into $\so(V)$. Since every map in
$\so(V)$ is a linear combination of the linear maps above, and since
$\dim(L)=\dim(\so(V))$, the map $L \to \so(V)$ is an isomorphism. We will
identify $\so(V)$ with the Lie subalgebra $L \subseteq \Cl(V)$ via the
inverse of this isomorphism, and we will identify
$\Wedge^2 V$ with $\so(V)$ via the map $u\wedge v \mapsto ( w \mapsto
(v|w) u - (u|w)v)$. The concatenation of these identifications is the
linear map $\frac{1}{4} \psi$. 

\subsection{The half-spin representations} 
\label{ss:spin-rep}

From now on, we assume that $\dim(V)=2n$. We believe that 
all our results hold \textit{mutatis mutandis} also in the
odd-dimensional case, but we have not checked the details. 
A \textit{maximal} isotropic subspace $U$ of $V$ is an isotropic subspace which is maximal with respect to inclusion. Since $K$ is algebraically closed, $q$ has maximal Witt index, so that every maximal isotropic subspace of $V$ has dimension $n$.

The spin representation of $\so(V)$ is constructed as follows. Let $F$ be a maximal isotropic subspace of $V$ and let $f_1, \ldots, f_n$ be a basis of $F$. Define $f:=f_1 \cdots f_n \in \Cl(F)$; this element in $\Cl(F)=\Wedge F$ is well-defined up to a scalar. Then the left ideal $\Cl(V) \cdot f$ is a left module for the associative algebra $\Cl(V)$, and hence for its Lie subalgebra $\so(V)$. This ideal is called the \textit{spin representation} of $\so(V)$. As $\Cl(V)$ is $\Z/2\Z$-graded, the spin representation splits into a direct sum of two subrepresentations for $\Cl^+(V)$, and hence for $\so(V) \subseteq \Cl^+(V)$, namely, $\Cl^+(V) \cdot f$ and $\Cl^-(V) \cdot f$. These representations are called the \textit{half-spin representations} of $\so(V)$.

\subsection{Explicit formulas}
\label{ssec:ExplicitFormulas}

We will need more explicit formulas for the action of $\so(V)$ on the half-spin representations. To this end, let $E$ be another isotropic $n$-dimensional subspace of $V$ such that $V=E \oplus F$. Then the map 
\[ 
    \Wedge E=\Cl(E) \to \Cl(V)f,\quad \omega \mapsto \omega f 
\]
is a linear isomorphism, and we use it to identify $\Wedge E$ with the spin representation. We write $\rho:\so(V) \to \End(\Wedge E)$ for the corresponding representation. It splits as a direct sum of the half-spin representations $\rho_+:\so(V) \to \End(\Wedge^+ E)$ and $\rho_-:\so(V) \to \End(\Wedge^- E)$ where $\Wedge^+E=\bigoplus_{d \text{ even}} \Wedge^d E$ and $\Wedge^-E=\bigoplus_{d \text{ odd}} \Wedge^d E$.

In this model of the spin representation, the action of $v \in E \subseteq \Cl(V)$ on the spin representation $\Wedge E$ is just the outer product on $\bigwedge E: o(v):\Wedge E \to \Wedge E,\ \omega \mapsto v \wedge \omega$, while the action of $v \in F \subseteq \Cl(V)$ is twice the inner product on $\bigwedge E$:
\[ 
    2 \iota(v) w_1 \wedge \cdots \wedge w_k = 2 \sum_{i=1}^k (-1)^{i-1}(v|w_i) w_1 \wedge \cdots \wedge \^{w_i} \wedge \cdots \wedge w_k. 
\]
The factor $2$ and the alternating signs come from the following identity in $\Cl(V)$: 
\[ 
    v v_i=2 (v|v_i) - v_i v \text{ for } v \in F \text{ and } v_i \in E.
\]
For a general $v \in V$ we write $v = v' + v''$ with $v' \in E$, $v'' \in F$. Then the action of $V$ on $\Wedge E$ is given by
\[
    v \mapsto o(v') + 2 \iota(v'').
\]
We now compute the linear maps by means of which $\so(V)$ acts on $\Wedge E$. To this end, recall that a pair $e,f \in V$ is called \textit{hyperbolic} if $e$, $f$ are isotropic and $(e|f) = 1$. Given the basis $f_1,\ldots,f_n$ of $F$, there is a unique basis $e_1,\ldots,e_n$ of $E$ so that $(e_i|f_j)=\delta_{ij}$; then $e_1,\ldots,e_n,f_1,\ldots,f_n$ is called a \textit{hyperbolic basis} of $V$. Now the element $e_i \wedge e_j \in \so(V)$ acts on $\Wedge E \isom \Cl(V)f$ via the linear map
\[
    \frac{1}{4}(o(e_i)o(e_j)-o(e_j)o(e_i))=\frac{1}{2} o(e_i)o(e_j);
\] 
the element $f_i \wedge f_j$ acts via the linear map
\[ 
    \frac{1}{4}(4\iota(f_i) \iota(f_j)-4\iota(f_j)\iota(f_i))=2 \iota(f_i) \iota(f_j); 
\]
and the element $e_i \wedge f_j$ acts via the linear map
\[ 
    \frac{1}{4}(o(e_i) 2 \iota(f_j)- 2 \iota(f_j) o(e_i))=\frac{1}{2} (o(e_i) \iota(f_j)- \iota(f_j) o(e_i)).
\]
In particular, $\omega_0:=e_1 \wedge \cdots \wedge e_n \in \Wedge E$ is mapped to $0$ by all elements $e_i \wedge e_j$ and all elements $e_i \wedge f_j$ with $i \neq j$, and it is mapped to $\frac{1}{2} \omega_0$ by all $e_i \wedge f_i$. 

\subsection{Highest weights of the half-spin representations}
\label{ssec:highest weight}

Recall, e.g. from \cite[Chapter IV, pages 140--141]{Jacobson62}, that in the basis $e_1,\ldots,e_n,f_1,\ldots,f_n$, matrices in $\so(V)$ have the form 
\[ 
    \begin{bmatrix} A&B \\ C&-A^T \end{bmatrix} \text{ with } B^T=-B, \text{ and } C^T=-C. 
\]
Here the $(e_i,e_j)$-entry of $A$ is the coefficient of $e_i \wedge f_j$, the $(e_i,f_j)$-entry of $B$ is the coefficient of $e_i \wedge e_j$, and the $(f_i,e_j)$-entry of $C$ is the coefficient of $f_i \wedge f_j$. 

The diagonal matrices $e_i \wedge f_i$ span a Cartan subalgebra of $\so(V)$ with standard (Chevalley) basis consisting of $h_i:=e_i \wedge f_i - e_{i+1}
\wedge f_{i+1}$ for $i=1,\ldots,n-1$ and $h_n:=e_{n-1} \wedge f_{n-1}
+ e_n \wedge f_n$ (this last element is forgotten in the basis of the
Cartan algebra on \cite[page 140]{Jacobson62}).

Now $(e_i \wedge e_j) \omega_0=(e_i \wedge f_j) \omega_0=0$ for all $i \neq j$. Furthermore, the elements $h_1,\ldots,h_{n-1}$ map $\omega_0$ to $0$, while $h_n$ maps $\omega_0$ to $\omega_0$. Thus the Borel subalgebra maps the line $K \omega_0$ into itself and $\omega_0$ is a highest weight vector of the \textit{fundamental weight} $\lambda_0:=(0,\ldots,0,1)$ in the standard basis. Summarising, $\omega_0 \in \Wedge E$ generates a copy of the irreducible $\so(V)$-module $V_{\lambda_0}$ with highest weight $\lambda_0$. Clearly, the $\so(V)$-module generated by $\omega_0$ is contained in $\Wedge^+ E$ if $n$ is even, and contained in $\Wedge^- E$ when $n$ is odd. One can also show that both half-spin representations are irreducible, hence one of them is a copy of $V_{\lambda_0}$. For the other half-spin representation, consider the element 
\[
    \omega_1:=e_1 \wedge \cdots \wedge e_{n-1} \in \Wedge E. 
\] 
This element is mapped to zero by $e_i \wedge e_j$ for all $i \neq j$ and by $e_i \wedge f_j$ for all $i < j$. It is further mapped to $0$ by $h_1,\ldots,h_{n-2},h_n$, and to $\omega_1$ by $h_{n-1}$. For example, we have
\begin{align*} 
    h_n \omega_1 &= \frac{1}{2}\left(o(e_{n-1})\iota(f_{n-1})-\iota(f_{n-1})o(e_{n-1})+o(e_n) \iota(f_n)-\iota(f_n) o(e_n)\right) e_1 \wedge \cdots \wedge e_{n-1} \\
    &=\frac{1}{2}(1 - 0 + 0 - 1) \omega_1 = 0, \text{ and similarly}\\ h_{n-1} \omega_1&=\frac{1}{2} (1-0-0+1) \omega_1=\omega_1.
    \end{align*}
Hence $\omega_1$ generates a copy of $V_{\lambda_1}$, the irreducible $\so(V)$-module of highest weight $\lambda_1 := (0,\ldots,0,1,0)$; this is the other half-spin representation.

\subsection{The spin group}

Let $\rho: \so(V) \to \End(\Wedge E)$ be the spin representation. The
spin group $\Spin(V)$ can be defined as the subgroup of $\GL(\Wedge E)$
generated by the one-parameter subgroups $t \mapsto \exp(t \rho(X))$
where $X$ runs over the root vectors $e_i \wedge e_j, f_i \wedge f_j$
and $e_i \wedge f_j$ with $i \neq j$. Note that $\rho(X)$ is nilpotent
for each of these root vectors, so that $t \mapsto \exp(t \rho(X))$
is an {\em algebraic} group homomorphism $K \to \GL(\Wedge E)$. It is a
standard fact that the subgroup generated by irreducible curves through
the identity in an algebraic group is itself a connected algebraic
group; see \cite[Proposition 2.2]{Borel91}. So $\Spin(V)$ is a connected
algebraic group, and one verifies that its Lie algebra is isomorphic
to the Lie algebra generated by the root vectors $X$, i.e., to $\so(V)$.

By construction, the (half-)spin representations $\Wedge
E$, $\Wedge^+ E$ and $\Wedge^- E$ are representations of $\Spin(V)$.
We use the same notation
$
     \rho: \Spin(V) \to \GL(\Wedge E), \ \rho_+ \colon \Spin(V) \to
     \GL(\bigwedge\nolimits^+ E), \text{ and } \rho_- \colon
     \Spin(V) \to \GL(\bigwedge\nolimits^- E)
$
for these as for the corresponding Lie algebra representations.

\begin{remark}
The algebraic group $\Spin(V)$ is usually constructed as a subgroup
of the unit group $\Cl^\inv(V)$ as follows: consider first
\[
    \Gamma(V) = \{x\in \Cl^\inv(V) \mid xVx^{-1} = V\},
\]
sometimes called the Clifford group. Then $\Spin(V)$ is the subgroup
of $\Gamma(V)$ of elements of \emph{spinor norm} $1$; that is, $x x^* =
1$, where $x^*$ denotes the involution defined in \cref{ssec:Clifford}.
In this model of the spin group it is easy to see that it admits a
$2:1$ covering $\Spin(V) \to \SO(V)$, namely, the restriction of the
homomorphism $\Gamma(V)
\to \O(V)$ given by associating to $x \in \Gamma(V)$ the orthogonal
transformation $w \mapsto x w x^{-1}$. For more details see \cite{Pro07}.
Since our later computations involve the Lie algebra
$\so(V)$ only, the definition of $\Spin(V)$ above suffices for our
purposes.
\end{remark}

The half-spin representations are \textit{not} representations of the group $\SO(V)$; this can be checked, e.g., by showing that the highest weights $\lambda_0$ and $\lambda_1$ are not in the weight lattice of $\SO(V)$. 

\subsection{Two actions of \texorpdfstring{$\gl(E)$}{} on \texorpdfstring{$\Wedge E$}{}}
\label{ssec:two-actions-gl}

The definition of the (half-)spin representation(s) of $\so(V)$ and $\Spin(V)$ as $\Cl^{(\pm)}(V) f$ involves only the quadratic form $q$ and the choice of a maximal isotropic space $F \subseteq V$. Consequently, any linear automorphism of $V$ that preserves $q$ and maps $F$ into itself also acts on $\Cl^{(\pm)}(V) f$. These linear automorphisms form the stabiliser of $F$ in $\SO(V)$, which is the parabolic subgroup whose Lie algebra consists of the matrices in $\SO(V)$ that are block lower triangular in the basis $e_1,\ldots,e_n,f_1,\ldots,f_n$. So, while $\SO(V)$ does not act naturally on the (half-)spin representation(s), this stabiliser does.

In particular, in our model $\Wedge^{(\pm)} E$ of the (half-)spin representation(s), the group $\GL(E)$, embedded into $\SO(V)$ as the subgroup of block diagonal matrices
\[ 
    \begin{bmatrix} a & 0 \\ 0 & -a^{T} \end{bmatrix} 
\]
acts on $\Wedge E$ in the natural manner. We stress that this is \textit{not} the action obtained by integrating the action of $\gl(E) \subseteq \so(V)$ on $\Wedge E$ \textit{regarded as the spin representation}. Indeed, the standard action of $e_i \wedge f_j \in \gl(E)$ on $\omega:=e_{i_1} \wedge \cdots \wedge e_{i_k} \in \Wedge^k E$ yields
\[ 
    \sum_{l=1}^k e_{i_1} \wedge \cdots \wedge e_i (f_j|e_{i_l}) \wedge \cdots \wedge e_{i_k} =
    \begin{cases}
        0 \text{ if $j \not \in \{i_1,\ldots,i_k\}$}\\(-1)^{l-1} e_i \wedge e_{i_1} \wedge \cdots \wedge \widehat{e_{i_l}} \wedge \cdots \wedge e_{i_k} \text{ if $j=i_l$.} 
    \end{cases}
\]
On the other hand, in the spin representation the action is given by the linear map $\frac{1}{2}(o(e_i) \iota(f_j) - \iota(f_j) o(e_i))$. If $j \neq i$ and $j \not \in \{i_1,\ldots,i_k\}$, then 
\[
    o(e_i) \iota(f_j) \omega = \iota(f_j) o(e_i) \omega = 0.
\]
If $j \neq i$ and $j=i_l$, then 
\[
    o(e_i) \iota(f_j) \omega = (-1)^{l-1} e_i \wedge e_{i_1} \wedge \cdots \wedge \^{e_{i_l}} \wedge \cdots \wedge e_{i_k}  = - \iota(f_j) o(e_i) \omega. 
\]
We conclude that for $i \neq j$, the action of $e_i \wedge f_j$ is the same in both representations. However, if $i=j$, then 
\[
    \frac{1}{2}(o(e_i) \iota(f_i) - \iota(f_i)o(e_i)) \omega = \begin{cases}
        -\frac{1}{2} \omega \text{ if $i \not \in \{i_1,\ldots,i_k\}$, and}\\ \frac{1}{2} (-1)^{l-1} e_i \wedge e_{i_1} \wedge \cdots \wedge \widehat{e_{i_l}} \wedge \cdots \wedge e_{i_k} = \frac{1}{2} \omega \text{ if $i=i_l$}.
    \end{cases}
\]
We conclude that if $\tilde{\rho}:\gl(E) \to \End(\Wedge E)$ is the standard representation of $\gl(E)$, then the restriction of the spin representation $\rho:\so(V) \to \End(\Wedge E)$ to $\gl(E)$ \textit{as a subalgebra of $\so(V)$} satisfies
\begin{equation}\label{re:Twist} 
    \rho(A)=\tilde{\rho}(A) - \frac{1}{2} \tr(A) \Id_{\Wedge E}.
\end{equation}

At the group level, this is to be understood as follows. The pre-image of $\GL(E) \subseteq \SO(V)$ in $\Spin(V)$ is isomorphic to the connected algebraic group 
\[ 
    H:=\big\{(g,t) \in \GL(E) \times K^* \mid \det(g)=t^2\big\} 
\]
for which $(g,t) \mapsto g$ is a $2:1$ cover of $\GL(E)$, and the restriction of $\rho$ to $H$ satisfies $\rho(g,t)=\tilde{\rho}(g) \cdot t^{-1}$---a ``twist of the standard representation by the inverse square root of the determinant''. 

\section{The isotropic Grassmannian and infinite spin representations}
\label{sec:IsoGras/InfiniteSpin}

\subsection{The isotropic Grassmannian in its spinor embedding} \label{ssec:igc}

As before, let $V$ be a $2n$-dimensional vector space over
$K$ endowed with a nondegenerate quadratic form. The (maximal)
\emph{isotropic Grassmannian} $\Gr_{\iso}(V,q)$ parametrizes all maximal
isotropic subspaces of $V$. It has two connected components, denoted
$\Gr^+_{\iso}(V)$ and $\Gr^-_{\iso}(V)$. The goal of this subsection is
to introduce the isotropic Grassmann cone, which is an affine cone
over $\Gr_{\iso}(V,q)$ in the spin representation. 

Fix a maximal isotropic subspace $F \subseteq V$ and as before set
$f\coloneqq f_1 \cdots f_n \in \Cl(V)$, where $f_1, \dots, f_n$ is
any basis of $F$. Now let $H \subseteq V$ be another maximal isotropic
space. Then we claim that the space 
\begin{equation}\label{eq:S_H}
    S_H := \{\omega \in \Cl(V)f \mid v \cdot \omega = 0 \text{ for
    all } 
    v \in H\} \subseteq \Cl(V)f
\end{equation}
is $1$-dimensional. Indeed, we may find a hyperbolic
basis $e_1,\ldots,e_n,f_1,\ldots,f_n$ of $V$ such that
$f_1,\ldots,f_k$ span $H \cap F$, $f_1,\ldots,f_n$ span $F$, and
$e_{k+1},\ldots,e_n,f_{1},\ldots,f_k$ span $H$. We call this hyperbolic
basis {\em adapted to $H$ and $F$}. Then the element 
\[ \omega_H:=e_{k+1} \cdots e_n f_1 \cdots f_k f_{k+1} \cdots f_n \in
\Cl(V)f \]
lies in $S_H$ since $e_i \omega_H=f_j \omega_H = 0$ for all $i > k$ and
$j  \leq k$. Conversely, if $\mu \in S_H$, then write 
\[ \mu=\sum_{l=0}^n \sum_{i_1<\ldots<i_l} c_{\{i_1,\ldots,i_l\}} e_{i_1} \cdots e_{i_l}
f. \]
If $c_I \neq 0$ for some $I$ with $I \not \supseteq \{k+1,\ldots,n\}$,
then for any $j \in \{k+1,\ldots,n\} \setminus I$ we find that $e_j \mu
\neq 0$. So all $I$ with $c_I \neq 0$ contain $\{k+1,\ldots,n\}$. If some
$I$ with $c_I \neq 0$ further contains an $i \leq k$, then $f_i \mu$
is nonzero. Hence $S_H$ is spanned by $\omega_H$, as claimed. In what
follows, by slight abuse of notation, we will write $\omega_H$ for any
nonzero vector in $S_H$. 

The space $H$ can be uniquely recovered from $\omega_H$ via 
\[ H=\{v \in V \mid v \cdot \omega_H = 0 \}. \]
Indeed, we have already seen $\subseteq$. For the converse, observe
that the vectors $e_i \omega_H, f_j \omega_H$ with $i \leq k$ and $j > k$ are
linearly independent. 

The map that sends $H \in \Gr_{\iso}(V,q)$ to the projective point representing it, i.e.,
\[
    H \mapsto [\omega_H] \in \P(\Cl(V)f),
\]
is therefore injective, and it is called the \textit{spinor embedding}
of the isotropic Grassmannian (see \cite{manivel2009spinor}). The
\textit{isotropic Grassmann cone} is defined as
\[
    \widehat{\Gr}_{\iso}(V,q) \coloneqq \bigcup_{H} \lspan{\omega_H} \subseteq \Cl(V)f,
\]
where the union is the taken over all maximal isotropic subspaces
$H\subseteq V$. We denote by $\widehat{\Gr}^\pm_{\iso}(V,q) \coloneqq
\widehat{\Gr}_{\iso}(V,q) \cap \Cl^\pm(V)f$ the cones over the connected components of the isotropic Grassmannian in its spinor embedding. 

\subsection{Contraction with an isotropic vector} 
\label{ssec:Contraction}

Let $e \in V$ be a nonzero isotropic vector. Then $V_e:=e^\perp / \lspan{e}$ is equipped with a natural nondegenerate quadratic form, and there is a rational map $\Gr_{\iso}(V) \to \Gr_{\iso}(V_e)$ that maps an $n$-dimensional isotropic space $H$ to the image in $V_e$ of the $(n-1)$-dimensional isotropic space $H \cap e^\perp$ (this is defined if $e \not \in H$, which by maximality of $H$ is equivalent to $H \not \subseteq e^\perp$). This map is the restriction to $\Gr_{\iso}(V)$ of the rational map $\PP(\Wedge^n V) \to \PP(\Wedge^{n-1} V_e)$ induced by the linear map (``contraction with $e$''):
\[
    c_e: \Wedge^n V \to \Wedge^{n-1} V_e, \quad v_1 \wedge \cdots \wedge v_n \mapsto \sum_{i=1}^n (-1)^{i-1} (e|v_i) \ol{v_1} \wedge \cdots \wedge \^{v_i} \wedge \cdots \wedge \ol{v_n} 
\] 
where $\ol{v_i}$ is the image of $v_i$ in $V/\lspan{e}$. Note first that this map is the inner product $\iota(e)$ followed by a projection. Furthermore, \textit{a priori}, the codomain of this map is the larger space $\Wedge^{n-1} (V/\lspan{e})$, but one may choose $v_1,\ldots,v_n$ such that $(e|v_i)=0$ for $i>1$, and then it is evident that the image is indeed in $\Wedge^{n-1} V_e$.

We want to construct a similar contraction map at the level of the
spin representation. For reasons that will become clear in a moment,
we restrict our attention first to a map between two half-spin
representations, as follows. Assume that $e \notin F$, and choose a basis $f_1,\ldots,f_n$ of $F$ such that $(e|f_i)=\delta_{in}$. As usual, write $f:=f_1 \cdots f_n$, and write $\ol{f}:=\ol{f}_1 \cdots \ol{f}_{n-1}$, so that $\Cl^+(V_e)\ol{f}$ is a half-spin representation of $\so(V_e)$. 

Then we define the map
\[
    \pi_e:\Cl^+(V)f \to \Cl^+(V_e)\ol{f}, \quad 
    \pi_e(a f):=\text{ the image of } \frac{1}{2}((-1)^{n-1} eaf+afe) \text{ in } \Cl(V_e)\ol{f},
\]
where the implicit claim is that the expression on the right lies in $\Cl(e^\perp)f_1\cdots f_{n-1}$, so that its image in $\Cl(V_e) \ol{f}$ is well-defined (note that the projection $e^\perp \to V_e$ induces a homomorphism of Clifford algebras), and that this image lies in the left ideal generated by $\ol{f}$. To verify this claim, and to derive a more explicit formula for the map above, let $e_1,\ldots,e_{n}=e$ be a basis of an isotropic space $E$ complementary to $F$. Then it suffices to consider the case where $a=e_{i_1} \cdots e_{i_k}$ for some $i_1<\ldots<i_k$. We then have 
\begin{align*} 
    eaf&=ee_{i_1} \cdots e_{i_k} f_1\cdots f_n\\
    &=\begin{cases} 0 \text{ if $i_k=n$, and}\\ 2 (-1)^{k+n-1}e_{i_1} \cdots e_{i_k} f_1 \cdots f_{n-1} +(-1)^{k+n}e_{i_1} \cdots e_{i_k} f_1 \cdots f_n e \text{ otherwise.} \end{cases}
\end{align*}
Multiplying by $(-1)^{n-1}$ and using that $k$ is even, the latter expression becomes
\[ 
    2 e_{i_1} \cdots e_{i_k} f_1 \cdots f_{n-1} - a f e. 
\]
Hence we conclude that
\[
    \pi_e(e_{i_1} \cdots e_{i_k} f)=
    \begin{cases}
        0 & \text{if $i_k=n$, and}\\
        \ol{e}_{i_1} \cdots \ol{e}_{i_k} \ol{f} & \text{otherwise.}
    \end{cases}
\]
In short, in our models $\Wedge^+ E$ and $\Wedge^+(E/\lspan{e})$ for the half-spin representations of $\so(V)$ and $\so(V_e)$, $\pi_e$ is just the reduction-mod-$e$ map. We leave it to the reader to check that the reduction-mod-$e$ map $\Wedge^- E \to \Wedge^-(E/\lspan{e})$ arises in a similar fashion from the map 
\[
    \pi_e:\Cl(V)^-f \to \Cl(V_e)^-\ol{f}, \quad \pi_e(a f):=\text{ the image of } \frac{1}{2}((-1)^{n} eaf+afe) \text{ in } \Cl(V_e)\ol{f}.
\]
We will informally call the maps $\pi_e$ ``contraction with $e$''. Together they define a map on $\Cl(V)f$ which we also denote by $\pi_e$.

\begin{proposition} \label{prop:pihom}
The contraction map $\pi_e:\Cl(V)f \to \Cl(V)\overline{f}$ is a 
homomorphism of $\Cl(e^\perp)$-representations.
\end{proposition}

\begin{proof}
Let $v \in e^\perp$ and consider $a \in \Cl^-(V)$. Then $va \in
\Cl^+(V)$ and hence $\pi_e(vaf)$
is the image in $\Cl(V_e)\overline{f}$ of 
\[ \frac{1}{2}((-1)^{n-1} evaf + vafe) = 
\frac{1}{2}((-1)^{n} veaf + vafe) = 
v \frac{1}{2}((-1)^n eaf + afe),
\]
where we have used $(v|e)=0$ in the first equality. The right-hand
side clearly equals $\overline{v}$ times the image of $\pi_e(af)$
in $\Cl(V_e) \overline{f}$.
\end{proof}

\subsection{Multiplying with an isotropic vector}
\label{ssec:Multiplication}

In a sense dual to the contraction maps $c_e:\Wedge^n V \to
\Wedge^{n-1} V_e$ are multiplication maps defined as follows. Let
$e,h \in V$ be isotropic with $(e|h)=1$; such a pair is called a
\textit{hyperbolic pair}. We then have $V=\langle e,h \rangle \oplus
\langle e,h \rangle^\perp$, and the map from the second summand to
$V_e=e^\perp/\langle e \rangle$ is an isometry. We use this isometry
to identify $V_e$ with the subspace $\langle e,h \rangle^\perp$ of
$V$ and write $s_e$ for the corresponding inclusion map.
Then we define 
\[ 
    m_h:\Wedge^{n-1} V_e \to \Wedge^n V, \quad \ol{v}_1 \wedge \cdots \wedge \ol{v}_{n-1} \mapsto h \wedge \ol{v}_1 \wedge \cdots \wedge \ol{v}_{n-1},
\]
which is just the outer product $o(h)$. The projectivisation of this map sends $\Gr_{\iso}(V_e)$ isomorphically to the closed subset of $\Gr_{\iso}(V)$ consisting of all $H$ containing $h$. We further observe that 
\[
    c_e \circ m_h=\id_{\Wedge^{n-1} V_e}. 
\]
We define a corresponding multiplication map at the level of spin
representations as follows: first, we assume that $h \in F$, and choose a basis $f_1,\ldots,f_n=h$ of $F$ such that $(e|f_i)=\delta_{in}$. As usual, we set $f=f_1 \cdots f_n$ and $\ol{f}=\ol{f}_1 \cdots \ol{f}_{n-1}$. Then we define
\[
    \tau_h:\Cl(V_e)\ol{f} \to \Cl(V) f,\quad \tau_h(a \ol{f}) := a \ol{f}f_n=a f.
\]
Note that, for $a \in \Cl(V_e)$, we have
\[ 
    \pi_e(\tau_h (a \ol{f}))=\pi_e(a f)=a \ol{f},
\]
where the last identity can be seen verified in the model $\Wedge E$
for the spin representation, where $\pi_e$ is the reduction-mod-$e$ map,
and $\tau_h$ is just the inclusion $\Wedge E/\langle e \rangle \to \Wedge
E$ corresponding to the inclusion $V_e \to V$.
So $\pi_e \circ \tau_h=\id_{\Cl(V_e)\ol{f}}$. We
will informally call $\tau_h$ the multiplication map with $h$.

\begin{proposition} \label{prop:tauhom}
The multiplication map $\tau_h:\Cl(V_e)\ol{f} \to \Cl(V)f$ is a homomorphism of
$\Cl(V_e)$-representations, where $\Cl(V_e)$ is regarded a subalgebra
of $\Cl(V)$ via the section $s_e:V_e \to V$.
\end{proposition}

\begin{proof}
Let $v \in V_e$ and let $a \in \Cl(V_e)$. Then 
\[ \tau_h(va\ol{f}) = va\ol{f}f_n = vaf, \]
as desired. 
\end{proof}

\begin{corollary}
Both the map $\pi_e: \Cl(V)f \to \Cl(V_e)\ol{f}$ and the map $\tau_h
:\Cl(V_e)\ol{f} \to \Cl(V)f$ are $\Spin(V_e)$-equivariant, where
$\Spin(V_e)$ is regarded as a subgroup of $\Spin(V)$ via the orthogonal
decomposition $V=V_e \oplus \langle e,h \rangle$.
\end{corollary}

\begin{proof}
Propositions~\ref{prop:pihom} and~\ref{prop:tauhom} imply that both
maps are homomorphisms of $\so(V_e)$-representations. Since
$\Spin(V_e)$ is generated by one-parameter subgroups corresponding to
nilpotent elements of $\so(V_e)$, $\pi_e$ and $\tau_h$ are $\Spin(V_e)$-equivariant.
\end{proof}

\subsection{Properties of the isotropic Grassmannian}

The goal of this subsection is to collect properties of the isotropic
Grassmann cone that will later motivate the definition of a \textit{(half-)spin variety} (see \Cref{s:spin-vars}). We fix a maximal isotropic
subspace $F \subseteq V$ and a hyperbolic pair $(e,h)$ with $h \in
F$ and $e \not \in F$ and identify $V_e=e^\perp/\langle e \rangle$
with the subspace $\langle e,h \rangle^\perp$ of $V$. We choose any
basis $f_1,\ldots,f_n$ of $F$ with $f_n=h$ and $(e|f_i)=0$ for $i<n$
and write $f:=f_1 \cdots f_n \in \Cl(V)$ and $\ol{f}:=\overline{f_1}
\cdots \overline{f_{n-1}} \in \Cl(V_e)$.

\begin{proposition}\label{prop:properties}

The isotropic Grassmann cone in $\Cl(V)f$ has the following
properties:
    \begin{enumerate}
        \item $\widehat{\Gr}_{\iso}(V) \subseteq \Cl(V)f$ is Zariski-closed and $\Spin(V)$-stable.

        \item \label{i2:properties} Let $\pi_e: \Cl(V) f \to \Cl(V_e) \ol{f}$ be the
	contraction defined in \S\ref{ssec:Contraction}. Then for
	every maximal isotropic subspace $H \subseteq V$ we have 
        \[
             \pi_e(S_H) \subseteq S_{H_e},
        \]
	where $H_e \subseteq V_e$ is the image of $e^\perp \cap H$
	in $V_e$.

        \item Let $\tau_h: \Cl(V_e) \ol{f} \to \Cl(V) f$ be the map defined in \S\ref{ssec:Multiplication}. Then for every maximal isotropic $H' \subseteq V_e$ we have
        \[
             \tau_h(S_{H'}) = S_{{H'} \oplus \lspan{h}}.
        \]    

    \end{enumerate}
    In particular, the contraction and multiplication map $\pi_e$ and
    $\tau_h$ preserve the isotropic Grassmann cones, i.e., 
    \[
        \pi_e\big(\widehat{\Gr}_{\iso}(V)\big) \subseteq
	\widehat{\Gr}_{\iso}(V_e) \quad \text{and} \quad
	\tau_h\big(\widehat{\Gr}_{\iso}(V_e)\big) \subseteq \widehat{\Gr}_{\iso}(V).
    \]   
\end{proposition}

\begin{proof}[Proof of \Cref{prop:properties}]
\begin{enumerate}
\item This is well known. Indeed, the isotropic Grassmann cone is the
union of the cones over the two connected components, and these cones
are the union of $\{0\}$ with the orbits of the highest weight vectors
$\omega_0$ and $\omega_1$. These minimal orbits are always Zariski
closed. For more detail see \cite[Theorem 1, p.428]{Pro07}.

\item Let $\omega_H$ be a spanning element of $S_H$. Then for all $v
\in e^\perp \cap H$ we have 
\[ \ol{v} \cdot \pi_e(\omega_H) = \pi_e(v \cdot \omega_H) =
\pi_e(0)=0,
\]
where the first equality follows from Proposition~\ref{prop:pihom}.
Hence $\pi_e(\omega_H)$ lies in $S_{H_e}$.

\item Let $\omega_{H'}$ be a spanning element of $S_{H'}$.
Then for all $v \in H'$ we have 
\[ v \cdot \tau_h(\omega_{H'}) = \tau_h (v \cdot \omega_{H'}) = \tau_h
(0)=0 \]
where the first equality holds by Proposition~\ref{prop:tauhom}.
Furthermore, we have 
\[ h \cdot \tau_h(\omega_{H'}) = h \cdot \omega_{H'} f_n = 0, \]
where we used the definition of $\tau_h$ and $h=f_n$. Thus
$\tau_h(\omega_{H'})$ lies in $S_{{H'} \oplus \lspan{h}}$. The equality now
follows from the fact that $\tau_h$ is injective. \qedhere
\end{enumerate}

\end{proof}

\begin{remark}
If $h \in H$, then $H=H_e \oplus \lspan{h}$ and since $\pi_e \circ
\tau_h$ is the identity on $\Cl(V_e) \ol{f}$ we find that
\[ \pi_e(S_H)=\pi_e(\tau_h(S_{H_e}))=S_{H_e}, \]
i.e., equality holds in \eqref{i2:properties} of \cref{prop:properties}. Later we will see that
equality holds under the weaker condition that $e \not \in H$, while
$\pi_e(S_H)=\{0\}$ when $e \in H$. These statements can also be checked by 
direct computations, but some care is needed since for $e,H,F$ in general
position one cannot construct a hyperbolic basis adapted to $H$ and $F$
that moreover contains $e$.
\end{remark}

\subsection{The dual of contraction}
\label{ssec:dual}

Let $e \not \in F \subseteq V$ be an isotropic vector. We want to compute the dual of the contraction map $\pi_e: \Cl(V)f \to \Cl(V_e)\ol{f}$; indeed, we claim that this is essentially the map
\[
    \psi_e: \Cl(V_e) \ol{f} \to \Cl(V) f 
\]
defined by its restriction $\Cl^\pm(V_e) \ol{f} \to \Cl^\mp(V) f$ as
\[ 
    \psi_e (\ol{b} \cdot \ol{f}_1 \cdots \ol{f_{n-1}}):= \pm e b f_1 \cdots f_n, 
\]
where the sign is $+$ on $\Cl^+(V_e) \ol{f}$ and $-$ on $\Cl^-(V_e) \ol{f}$. The reason for the ``flip'' of the choice of half-spin representation in the dual will become obvious below. Observe that $\psi_e$ is well-defined and, given a basis $e_1,\ldots,e_n=e$ of an isotropic space complementary to $F$ such that $e_1,\ldots,e_n,f_1,\ldots,f_n$ is a hyperbolic basis, maps $\overline{e_J} \ol{f}$ to $e_{J \cup \{n\}} f$. 

\begin{proposition} \label{prop:PsiDual}
    The following diagram:
    \[
        \xymatrix{(\Cl(V_e) \ol{f})^* \ar[d]_{\isom} \ar[r]^{\pi_e^*} & (\Cl(V)f)^* \ar[d]^{\isom} \\
        \Cl(V_e) \ol{f} \ar[r]_{\psi_e} & \Cl(V)f}
    \]
    can be made commuting via a $\Spin(V_e)$-module isomorphism on the left vertical arrow and a $\Spin(V)$-module isomorphism on the right vertical arrow. 
\end{proposition}

\begin{remark}\label{rem:Flip in PsiDual}
    The statement of \cref{prop:PsiDual} holds true when replacing $\Cl(V) f$ by either one of the two half-spin representations by considering the correct ``flip''. For example, if $n = \dim F$ is even, and $e_1,\dots, e_n, f_1, \dots, f_n$ is a hyperbolic basis as above, then in the $\Wedge E$-model the correct grading is
    \[
        \xymatrix{\left(\Wedge^+E_{n-1}\right)^* \ar[d]_{\isom} \ar[r]^{\pi_{e_n}^*} & \left(\Wedge^+E_n\right)^* \ar[d]^{\isom} \\
        \Wedge^-E_{n-1} \ar[r]_{\psi_{e_n}} & \Wedge^+ E_n.}
    \]
\end{remark}

To prove \cref{prop:PsiDual} we consider the bilinear form $\beta$ on
the spin representation $\Cl(V)f$ defined as in \cite{Pro07} as follows:
for $af,bf \in \Cl(V)f$ it turns out that $(af)^*bf=f^*a^*bf$, where $*$
denotes the anti-automorphism from \S\ref{ssec:Clifford}, is a scalar
multiple of $f$. The scalar is denoted $\beta(af,bf)$.  We have the
following properties:

\begin{lemma}[{\cite[p. 430]{Pro07}}] Let $\beta$ be the bilinear form defined as above.
    \label{lem:procesi-pairing}
    \begin{enumerate}
        \item The form $\beta$ is nondegenerate and $\Spin(V)$-invariant.
        \item $\beta$ is symmetric if $n \equiv 0,1 \mod 4$, and it is skew-symmetric if $n \equiv 2,3 \mod 4$.
        \item The two half-spin representations are self-dual via $\beta$ if $n$ is even, and each is the dual of the other if $n$ is odd.
    \end{enumerate}
\end{lemma}

In the proof of \cref{prop:PsiDual} we will use a hyperbolic
basis $e_1,\ldots,e_n,f_1,\ldots,f_n$ with $e_n=e$. For a subset
$I=\{i_1<\ldots<i_k\} \subseteq [n]$ set $e_I:=e_{i_1} \cdots e_{i_k}
\in \Cl(E) \isom \Wedge E$, where $E$ is the span of the $e_i$. We have seen in \S\ref{ssec:ExplicitFormulas}
that the spin representation has as a basis the elements $e_I f$ with $I$
running through all subsets of $[n]$.

\begin{proof}[Proof of \cref{prop:PsiDual}]
    Consider the bilinear forms $\beta$ on $\Cl(V)f$ and $\beta_e$
    on $\Cl(V_e)\ol{f}$ as defined above. By \cref{lem:procesi-pairing}
    the spin representations $\Cl(V)f$ and $\Cl(V_e)\ol{f}$ are self-dual
    via $\beta$ and $\beta_e$, respectively. Thus it suffices to prove,
    for $a \in \Cl(V)$ and $b \in \Cl(e^\perp)$, that
    \[ 
        \beta_e(\pi_e(af),\ol{b}\ol{f})= \frac{(-1)^{n-1}}{2} \beta(af,\psi_e(\ol{b}\ol{f})). 
    \]
    We may
    assume that $a = e_I$, $b = e_J$ with $I \subseteq [n]$, $J
    \subseteq [n-1]$. 

    In the $\Wedge E$-model $\pi_e$ is the mod-$e$ map, and hence the
    left-hand side is zero if $n \in I$. If $n \not \in I$, then the
    left-hand side equals the coefficient of $\ol{f}$ in $\ol{f}^*
    \ol{e_I}^* \ol{e_J} \ol{f}$.  This is nonzero if and only if
    $[n-1]$ is the disjoint union of $I$ and $J$, and then it is
    $2^{n-1}$ times a sign
    corresponding to the number of swaps needed to move the factors
    $\ol{f_i}$ of $\ol{f}^*$ to just before the corresponding factor
    $\ol{e_i}$ in either $\ol{e_I}^*$ or $\ol{e_J}$.

Apart from the factor $\frac{(-1)^{n-1}}{2}$, the right-hand side is the coefficient of $f$ in $f^* e_I e_J e_n
    f$. This is nonzero if and only if $[n]$ is the disjoint union of
    the sets $\{n\},J,I$, and in that case it is $2^n$ times a sign corresponding
    to the number of swaps needed to move the factors $f_i$ of $f^*$ to the
    corresponding factor $e_i$ in either $e_I$ or $e_J$ or (in the
    case of $f_n$) to just before the factor $e_n$. The latter
    contributes $(-1)^{n-1}$, and apart from this factor the sign is
    the same as on the left-hand side.
\end{proof}

\subsection{Two infinite spin representations}
\label{sec:infinte half-spin representations}

Let $V_\infty$ be the countable-dimensional vector space with basis $e_1,
f_1, e_2, f_2, \dots$, and equip $V_\infty$ with the quadratic form for
which this is a hyperbolic basis, i.e., $(e_i|e_j)=(f_i|f_j)=0$ and
$(e_i|f_j)=\delta_{ij}$ for all $i,j$. We write $E_\infty$ and $F_\infty$
for the subspaces of $V_\infty$ spanned by the $e_i$ and the $f_i$,
respectively. 

Let $V_n$ be the subspace of $V_\infty$ spanned by $e_1, f_1, e_2,
f_2, \dots, e_n, f_n$, with the restricted quadratic form. We further
set $E_n:=V_n \cap E_\infty$ and $F_n:=V_n \cap F_\infty$. 
We define the \textit{infinite spin group} as $\Spin(V_\infty):=\varinjlim_n \Spin(V_n)$  where $\Spin(V_{n-1})$ is embedded into $\Spin(V_n)$ as the subgroup that fixes $\lspan{e_n,f_n}$ element-wise. Similarly, we write $\GL(E_\infty):=\varinjlim_n \GL(E_n)$ and $H$ for the preimage of $\GL(E_\infty)$ in $\Spin(V_\infty)$. We use the notation $\so(V_\infty)$ and $\gl(E_\infty)$ for the corresponding direct limits of the Lie algebras $\so(V_n)$ and $\gl(E_n)$. Here the direct limits are taken in the categories of abstract groups, and Lie algebras, respectively.

The previous paragraphs give rise to various $\Spin(V_{n-1})$-equivariant maps between the spin representations of $\Spin(V_{n-1})$ and $\Spin(V_n)$. First, contraction with $e_n$,
\[ 
    \pi_{e_n}:\Cl(V_n)f_1 \cdots f_n \to \Cl(V_{n-1}) f_1 \cdots f_{n-1}, 
\]
 and second, multiplication with $f_n$, 
\[
    \tau_{f_n}:\Cl(V_{n-1})f_1 \cdots f_{n-1} \to \Cl(V_n) f_1 \cdots f_n. 
\]
We have that these satisfy $\pi_{e_n} \circ \tau_{f_n} =\id$. Third, the map
\[ 
    \psi_{e_n}:\Cl(V_{n-1})f_1 \cdots f_{n-1} \to \Cl(V_n) f_1 \cdots f_n 
\]
that is dual to $\pi_{e_n}$ in the sense of \Cref{prop:PsiDual}.

\begin{definition}
    The \textit{direct (infinite) spin representation} is the direct limit of all spaces $\Cl(V_n) f_1 \cdots f_n$ along the maps $\psi_{e_n}$. The \textit{inverse (infinite) spin representation} is the inverse limit of all spaces $\Cl(V_n) f_1 \cdots f_n$ along the maps $\pi_{e_n}$.
\end{definition}

Since the maps $\psi_{e_n},\pi_{e_n}$ are $\Spin(V_{n-1})$-equivariant, both of these spaces are $\Spin(V_\infty)$-modules. As the dual of a direct limit is the inverse limit of the duals, and since the maps $\psi_{e_n}$ and $\pi_{e_n}$ are dual to each other by \Cref{prop:PsiDual}, the inverse spin representation is the dual space of the direct spin representation.

In our model $\Wedge E_n$ of $\Cl(V_n) f_1 \cdots f_n$, the map
$\psi_{e_n}$ is just the right multiplication 
\[
    \Wedge E_{n-1} \to \Wedge E_n,\; \omega \mapsto \omega \wedge e_n.
\]
Hence the direct spin representation has as a basis formal infinite products
\[ 
    e_{i_1} \wedge e_{i_2} \wedge \ldots =: e_I
\]
where $I=\{i_1<i_2<\ldots\}$ is a cofinite subset of $\N$. We will
write $\Wedge_\infty E_\infty$ for this countable-dimensional vector
space. The action of the Lie algebra $\so(V_\infty)$ of
$\Spin(V_\infty)$ on this space is given via the explicit formulas from
\S\ref{ssec:ExplicitFormulas}. In particular, the span of the $e_I$
with $|\N \setminus I|$ even (respectively, odd) is a
$\Spin(V_\infty)$-submodule, and $\Wedge_\infty E_\infty$ is the direct sum of these (irreducible) modules.

\begin{remark}
    The reader may wonder why we do not introduce the direct spin representation as the direct limit of all $\Cl(V)f_1\cdots f_n$ along the maps $\tau_{f_n}$. This would make the ordinary Grassmann algebra $\Wedge E_\infty$ a model for the direct spin representation, instead of the slightly more complicated-looking space $\Wedge_\infty E_\infty$. However, the maps dual to the $\tau_{f_n}$ correspond to contraction maps with $f_n \in F$, which we have not discussed and which interchange even and odd half-spin representations. We believe that our theorem below goes through for this different setting, as well, but we have not checked the details.
\end{remark}

\subsection{Four infinite half-spin representations}\label{subsec:half-spin repr}

Keeping in mind that the maps $\psi_{e_n}$ interchange the even and odd subrepresentations, we define the \textit{direct (infinite) half-spin representations} $\Wedge_\infty^\pm E_\infty$ to be the direct limit
\begin{equation*}
    \Wedge_\infty^\pm E_\infty = \varinjlim \left(\Wedge^\pm E_0 \to \Wedge^\mp E_1 \to \Wedge^\pm E_2 \to \Wedge^\mp E_3 \to \Wedge^\pm E_4 \to \cdots \right)
\end{equation*}
along the maps $\psi_{e_n}$. 
For the sake of readability we will abbreviate this by 
\begin{equation}\label{eq:def of direct spin rep}
    \Wedge_\infty^\pm E_\infty = \varinjlim_n \Wedge^{\pm (-1)^n}E_n,
\end{equation}
where $\pm (-1)^n$ denotes $\pm$ if $n$ is even and $\mp$ if $n$ is odd. In terms of the basis $e_I$ introduced in \S\ref{sec:infinte half-spin representations}, the half-spin representation $\Wedge_\infty^+ E_\infty$ is spanned by all $e_I$ with $|\N \setminus I|$ even, and $\Wedge_\infty^- E_\infty$ by those with $|\N \setminus I|$ odd. The \textit{inverse (infinite) half-spin representations} are defined as the duals of the direct (infinite) half-spin representations. Using the isomorphisms from \Cref{rem:Flip in PsiDual} we observe
\begin{equation}\label{eq:inverse half-spin rep = proj limit}
    \left(\Wedge^\pm_\infty E_\infty \right)^*=\varprojlim_n \left( \Wedge^{\pm (-1)^n}E_n \right)^\ast \isom \varprojlim_n \Wedge^\pm E_n.
\end{equation}
So the inverse (infinite) half-spin representations can be identified with the inverse limits of the half-spin representations $\Wedge^\pm E_n$ along the projections $\pi_{e_n}$.

We can enrich the inverse spin representation to an affine scheme whose coordinate ring is the symmetric algebra on $\Wedge_\infty E_\infty$, recalling the following remark. 

\begin{remark}\label{rem:enriching dual spaces}
Let $K$ be any field (not necessarily algebraically closed) and $W$ any $K$-vector space (not necessarily finite dimensional). Then there are canonical identifications 
\[
    W^\ast = {\rm Spec}\big( {\rm Sym}(W)\big) (K) \subseteq 
    \left\{\text{closed points in }  {\rm Spec}\big( {\rm Sym}(W)\big)\right\}.
\]
So ${\rm Spec}\big( {\rm Sym}(W)\big)$ can be seen as an enrichment of $W^\ast$ to an affine scheme. 
If $W$ is a linear representation for a group $G$, then $G$ acts via $K$-algebra automorphisms on $\Sym W$ and hence via $K$-automorphisms on the affine scheme corresponding to $W^*$. For $W = \Wedge^\pm_\infty E_\infty$, this construction extends the natural $\Spin(V_\infty)$-action on the vector space $\varprojlim_n \Wedge^\pm E_n \isom W^*$ to the corresponding affine scheme.
\end{remark}

By abuse of notation, we will write $\left(\Wedge_\infty E_\infty\right)^*$ also for the scheme itself, and similarly for the inverse half-spin representations $\left(\Wedge^\pm_\infty E_\infty \right)^*$. Later we will also write $\Wedge^\pm E_n$ for the affine scheme ${\rm Spec}\left({\rm Sym}\left( \Wedge^{\pm (-1)^n}E_n \right) \right)$ by identifying $\Wedge^\pm E_n \cong \left(\Wedge^{\pm(-1)^n}E_n\right)^*$ as in \Cref{eq:inverse half-spin rep = proj limit}. \newpage

\section{Noetherianity of the inverse half-spin representations}
\label{s:noetherian}

In this section we prove our main theorem.

\begin{theorem} \label{thm:Main}
    The inverse half-spin representation $(\Wedge_\infty^+ E_\infty)^*$ is topologically Noetherian with respect to the action of $\Spin(V_\infty)$. That is, every descending chain 
    \[
        \left(\Wedge_\infty^+ E_\infty\right)^* \supseteq X_1 \supseteq X_2 \supseteq \ldots 
    \]
    of closed, reduced $\Spin(V_\infty)$-stable subschemes
    stabilises, and the same holds for the other inverse half-spin representation.
\end{theorem}

Recall that the action of $\Spin(V_\infty)$ on the inverse half-spin representation (as an affine scheme) is given by $K$-automorphisms, as described in \cref{rem:enriching dual spaces}.
We write $R$ for the symmetric algebra on the direct spin representation $\Wedge_\infty E_\infty$, so the inverse spin representation is $\Spec(R)$. Similarly, we write $R^\pm$ for the symmetric algebras on the direct half-spin representations, so $R^\pm$ is the coordinate ring of $\varprojlim_n \Wedge^\pm E_n$, respectively. 

Let us briefly outline the proof strategy. We will proceed by induction on the minimal degree of an equation defining a closed subset $X$. Starting with such an equation $p$, we show that there exists a partial derivative $q := \frac{\partial p}{\partial e_I}$ such that the principal open $X[1/q]$ is topologically $H_n$-Noetherian, where $H_n$ is the subgroup of $\Spin(V_\infty)$ defined below. For that we use that the $H_n$-action corresponds to a  ``twist'' of the usual $\GL(E_\infty)$-action, as observed in \cref{ssec:two-actions-gl} (for the exact formula see \eqref{re:Twist}); this allows us to apply the main result of \cite{Eggermont22}. Finally, for those points which are contained in the vanishing set of the $\Spin(V_\infty)$-orbit of $q$ we can apply induction, as the minimal degree of a defining equation has been lowered by $1$.

\subsection{Shifting}

Let $G_n$ be the subgroup of $G$ that fixes $e_1,\ldots,e_n,f_1,\ldots,f_n$ element-wise. Note that $G_n$ is isomorphic to $G$; at the level of the Lie algebras the isomorphism from $G$ to $G_n$ is given by the map
\[ 
    \begin{bmatrix}
        A & B \\
        C & -A^T
    \end{bmatrix} \mapsto 
    \begin{bmatrix}
        0 & 0 & 0 & 0 \\
        0 & A & 0 & B \\
        0 & 0 & 0 & 0 \\
        0 & C & 0 & -A^T
    \end{bmatrix}
\]
where the widths of the blocks are $n,\infty,n,\infty$, respectively. We write $H_n$ for $H \cap G_n$, where $H \subseteq \Spin(V_\infty)$ is the subgroup corresponding to the subalgebra $\gl(E_\infty) \subseteq \so(V_\infty)$. Then $H_n$ is the pre-image in $\Spin(V_\infty)$ of the subgroup $\GL(E_\infty)_n \subseteq \GL(E_\infty)$ of all $g$ that fix $e_1,\ldots,e_n$ element-wise and maps the span of the $e_i$ with $i>n$ into itself. The Lie algebra of $H_n$ and of $\GL(E_\infty)_n$ consists of the matrices above on the right with $B=C=0$.

\subsection{Acting with the general linear group on \texorpdfstring{$E$}{E}}

For every fixed $k \in \Z_{\geq 0}$, the Lie algebra $\gl(E_\infty) \subseteq \so(V_\infty)$ preserves the linear space
\[
    \left(\Wedge_\infty E_\infty\right)_k:=\big\langle \{e_I : |\N \setminus I|=k\} \big\rangle, 
\]
and hence so does the corresponding subgroup $H \subseteq
\Spin(V_\infty)$. We let $R_{\leq \ell} \subseteq R$ be the subalgebra generated by the spaces $(\Wedge_\infty E_\infty)_k$ with $k \leq \ell$. Crucial in the proof of Theorem~\ref{thm:Main} is the following result.

\begin{proposition} \label{prop:GLE}
    For every choice of nonnegative integers $\ell$ and $n$, $\Spec(R_{\leq \ell})$ is topologically $H_n$-Noetherian, i.e., every descending chain 
    \[ 
        \Spec(R_{\leq \ell}) \supseteq X_1 \supseteq X_2 \supseteq \ldots
    \]
    of $H_n$-stable closed and reduced subschemes stabilizes. 
\end{proposition}

The key ingredient in the proof of \Cref{prop:GLE} is the main result of \cite{Eggermont22}. In order to apply their result we need to do some preparatory work. We will start with the following lemma.

\begin{lemma} \label{lm:GLE}
    Every $H_n$-stable closed subscheme of $\Spec(R_{\leq \ell})$ is also stable under the group $\GL(E_\infty)_n$ acting in the natural manner on $\Wedge_\infty E_\infty$ and its dual, and \textit{vice versa}.
\end{lemma}

\begin{proof}
    \Cref{re:Twist} implies that $\gl(E_\infty) \subseteq \so(V_\infty)$ acts on $\Wedge_\infty E_\infty$ via
    \[ 
        \rho(A)=\tilde{\rho}(A)-\frac{1}{2} \tr(A) \id_{\Wedge_\infty E_\infty} 
    \]
    where $\tilde{\rho}$ is the standard representation of $\gl(E_\infty)$ on $\Wedge_\infty E_\infty$. An $H_n$-stable closed subscheme $X$ of $\Spec(R_{\leq \ell})$ is given by an $H_n$-stable ideal $I$ in the symmetric algebra $R_{\leq \ell}$. Such an $I$ is then also stable under the action of the Lie algebra $\gl(E_\infty)_n$ of $H_n$ by derivations that act on variables in $\bigoplus_{k=0}^\ell (\Wedge_\infty E_\infty)_{k}$ via $\rho$. 
    
    We claim that $I$ is a homogeneous ideal. Indeed, for $f \in I$,
    choose $m>n$ such that all variables in $f$ (which are basis
    elements $e_I$) contain the basis element $e_m$ of $E_\infty$. Let
    $A \in \gl(E_\infty)_n$ be the diagonal matrix with $0$'s
    everywhere except a $1$ on position $(m,m)$. Then $\rho(A)$ maps
    each variable in $f$ to $\frac{1}{2}$ times itself. Hence, by the
    Leibniz rule, $\rho(A)$ scales 
    the homogeneous part of degree $d$ in $f$ by $\frac{d}{2}$.
    Since $I$ is preserved by $\rho(A)$, it follows that $I$ contains all homogeneous components of $f$, and hence $I$ is a homogeneous ideal. 
    
    Now let $B \in \gl(E_\infty)_n$ and $f \in I$ be arbitrary. By the
    previous paragraph we can assume $f$ to be homogeneous of degree
    $d$, and we then have
    \[
        \rho(B)f = \tilde{\rho}(B)f - \frac{d}{2}\tr(B)f,
    \]
    and since $I$ is $\rho(B)$-stable, we deduce $\tilde{\rho}(B) f
    \in I$. This completes the proof in one direction. The proof in the 
    opposite direction is identical.
\end{proof}

\begin{remark} \label{re:Cone}
By the proof above,
    any $\Spin(V_\infty)$-stable closed subscheme $X$ of $(\Wedge_\infty E_\infty)^*$ is an affine cone.
\end{remark}

Following \cite{Eggermont22} the \textit{restricted dual} $(E_\infty)_*$ of $E_\infty$ is defined as the union $\bigcup_{n\geq1}(E_n)^*.$ We will denote by $\varepsilon^1, \varepsilon^2, \dots$ the basis of $(E_\infty)_*$ that is dual to the canonical basis $e_1, e_2, \dots$ of $E_\infty$ given by $\varepsilon^i(e_j)=\delta_{ij}$.

\begin{lemma} \label{lm:Duality}
    
    There is an $\SL(E_\infty)$-equivariant isomorphism
    \[
        \bigwedge \nolimits_\infty E_\infty \longrightarrow \bigwedge (E_\infty)_*,
    \]
    which restricts to an isomorphism
    \[
        \left( \bigwedge \nolimits_\infty E_\infty\right)_k \longrightarrow \bigwedge \nolimits^k(E_\infty)_*.
    \]
\end{lemma}

We will use this isomorphism to regard $\Wedge_\infty E_\infty$ as the
restricted dual of the Grassmann algebra $\Wedge E_\infty$. We stress,
though, that this isomorphism is not $GL(E_\infty)$-equivariant. 

\begin{proof}
    We have a natural bilinear map
    \[ 
        \Wedge E_\infty \times \Wedge_\infty E_\infty \to \Wedge_\infty E_\infty, \quad  (\omega,\omega') \mapsto \omega \wedge \omega'. 
    \]
    If $I \subseteq \N$ is finite and $J \subseteq \N$ is cofinite, then $e_I \wedge e_J$ is $0$ if $I \cap J \neq \emptyset$ and $\pm e_{I \cup J}$ otherwise, where the sign is determined by the permutation required to order the sequence $I,J$. We then define a perfect pairing $\gamma$ between the two spaces by
    \[
        \gamma(\omega,\omega'):=\text{the coefficient of $e_\N$ in } \omega \wedge \omega'.
    \]
   The map $\Phi_\gamma: \Wedge_\infty E_\infty \to \Wedge (E_\infty)_*, \; \omega' \mapsto \gamma(\cdot, \omega')$ induced by $\gamma$ is the isomorphism given by $e_I \mapsto \pm \varepsilon^{I^c}$, where $I^c \subseteq \N$ is the complement of $I$ and $\varepsilon^J \coloneqq \varepsilon^{j_1}\wedge \cdots \wedge \varepsilon^{j_k}$ for a finite set $J=\{j_1, \dots, j_k\}$. Note that $\gamma(A\cdot\omega, A\cdot \omega') = \det(A)\gamma(\omega, \omega')$ for all $A \in \GL(E_\infty)$, and hence $\gamma$ is $\SL(E_\infty)$-invariant. Therefore, the isomorphism $\Phi_\gamma$ is $\SL(E_\infty)$-equivariant. 
\end{proof}
   
\begin{lemma} \label{lem:SL-GL-invariant}
    An ideal $I \subseteq \Sym(\bigwedge(E_\infty)_*)$ is $\SL(E_\infty)$-stable if and only if it is $\GL(E_\infty)$-stable. The same holds for $\SL(E_\infty)_n$ and  $\GL(E_\infty)_n$.
\end{lemma}

\begin{proof}
    Assume that $I$ is $\SL(E_\infty)$-stable. Let $f \in I$ and $A \in \GL(E_\infty)$ be arbitrary. Choose $m=m(f,A) \in \N$ large enough so that $f \in \Sym(\Wedge(E_m)^*)$ and $A$ is the image of some $A_m \in \GL(E_m)$. Define $A_{m+1} \in \GL(E_{m+1})$ as the map given by $A_{m+1}(e_i)=A_m(e_i)$ for $i \leq m$ and $A_{m+1}(e_{m+1})=(\det(A_m))^{-1}( e_{m+1})$, and let $A^\prime$ be the image of $A_{m+1}$ in $\SL(E_\infty)$. Then the action of $A_m$ and $A_{m+1}$ agree on $(E_m)^*$. Hence they also agree on $\Sym(\Wedge(E_m)^*)$. So $A \cdot f = A^\prime \cdot f \in I$ since $I$ was assumed to be $\SL(E_\infty)$-stable and $A^\prime \in \SL(E_\infty)$. As $f \in I$ and $A \in \GL(E_\infty)$ were arbitrary, this shows that $I$ is $GL(E_\infty)$-stable.
\end{proof}

\begin{proof}[Proof of Proposition~\ref{prop:GLE}.]
    First, we claim that $\Spec\big(\Sym\big(\bigoplus_{k=0}^\ell \Wedge^k (E_\infty)_*\big)\big)$ is topologically $\GL(E_\infty)_n$-Noetherian. Indeed, the standard $\GL(E_\infty)$-representation of the space $\bigoplus_{k=0}^\ell \Wedge^k (E_\infty)_*$ is an algebraic representation and this also remains true when we act with $\GL(E_\infty)$ via its isomorphism into $\GL(E_\infty)_n$. Hence, the claim follows from \cite[Theorem 2]{Eggermont22}.
    Let $(X_i)_{i\in\N} \subseteq \Spec(R_{\leq \ell})$ be a descending chain of $H_n$-stable, closed, reduced subschemes. By \Cref{lm:GLE} every $X_i$ is also $\GL(E_\infty)_n$-stable. By \Cref{lm:Duality} there is an $\SL(E_\infty)_n$-equivariant isomorphism $\Spec(R_{\leq \ell}) \cong \Spec\big(\Sym\big(\bigoplus_{k=0}^\ell \Wedge^k (E_\infty)_*\big)\big)$. Let $X_i^\prime \subseteq \Spec\big(\Sym\big(\bigoplus_{k=0}^\ell \Wedge^k (E_\infty)_*\big)\big)$ be the closed, reduced, $\SL(E_\infty)$-stable subscheme corresponding to $X_i$ under this isomorphism. Using \Cref{lem:SL-GL-invariant} we see that the subschemes $X_i^\prime$ are also $\GL(E_\infty)_n$-stable. Therefore, the chain $(X_i^\prime)_{i\in \N}$ stabilizes by our first claim. Consequently, also the chain $(X_i)_{i\in\N}$ stabilizes. 
\end{proof}

Before we come to the proof of \cref{thm:Main}, let us recall the action of $f_i \wedge f_j \in \so(V_\infty)$ on $\Wedge^+_\infty E_\infty$ and its symmetric algebra $R^+$ in explicit terms. Recall from \cref{sec:infinte half-spin representations} that a basis for $\Wedge^+_\infty E_\infty$ is given by $e_I = e_{i_1} \wedge e_{i_2} \wedge \cdots$, where $I = \{i_1 < i_2 < \cdots\} \subseteq \N$ is cofinite and $|\N \setminus I|$ even. Then we have
\[
    (f_i \wedge f_j) e_I = \begin{cases}
        (-1)^{c_{i,j}(I)}e_{I\setminus \{i, j\}}& \text{if }
	i,j \in I, \text{ and}\\
        0 & \text{otherwise,}
    \end{cases}
\]
where $c_{i,j}(I)$ depends on the position of $i,j$ in $I$. (Note
that there is no factor $4$, since in our identification of $\Wedge^2
V$ to the Lie subalgebra $L$ of $\Cl(V)$ we had a factor
$\frac{1}{4}$.) The corresponding action of $f_i \wedge f_j$ on
polynomials in $R^+$ is as a derivation. 

\subsection{Proof of \texorpdfstring{\cref{thm:Main}}{Theorem~\ref{thm:Main}}}

Let $R^+ \subseteq R$ be the symmetric algebra on the direct half-spin representation $\Wedge^+_\infty E_\infty$, so that $\Spec(R^+)$ is the inverse half-spin representation $(\Wedge^+_\infty E_\infty)^*$. We prove topological $\Spin(V_\infty)$-Noetherianity of $\Spec(R^+)$; the corresponding statement for $\Spec(R^-)$ is proved in exactly the same manner. 

For a closed, reduced $\Spin(V_\infty)$-stable subscheme $X$ of
$\Spec(R^+)$, we denote by $\delta_X \in \{0,1,2,\ldots,\infty\}$ the
lowest degree of a nonzero polynomial in the ideal $I(X) \subseteq
R^+$ of $X$. Here we consider the natural grading on $R^+ = \Sym(\Wedge^+_\infty E_\infty)$, where the elements of $\Wedge^+_\infty E_\infty$ all have degree $1$.

We proceed by induction on $\delta_X$ to show that $X$ is topologically Noetherian; we may therefore assume that this is true for all $Y$
with $\delta_Y<\delta_X$. We have $\delta_X=\infty$ if and only if $X=\Spec(R^+)$. Then a chain
\[
    \Spec(R^+) = X \supseteq X_1 \supseteq X_2 \supseteq \ldots
\]
of $\Spin(V_\infty)$-closed subsets is either constant or else
there exists an $i$ with $\delta_{X_i}<\infty$. Hence it suffices
to prove that $X$ is Noetherian under the additional assumption
that $\delta_X<\infty$. At the other extreme, if $\delta_X=0$,
then $X$ is empty and there is nothing to prove. So we assume that $0<\delta_X<\infty$ and that all $Y$ with $\delta_Y<\delta_X$ are $\Spin(V_\infty)$-Noetherian.

Let $p \in R^+$ be a nonzero polynomial in the ideal of $X$ of degree $\delta_X$. By Remark~\ref{re:Cone}, since $X$ is a cone, $p$ is in fact homogeneous of degree $\delta_X$. 
Let $e_I$ be a variable appearing in $p$ such that $k := |I^c|$ is
maximal among all variables in $p$; note that $k$ is even. 
Then choose $n \geq k+2$ even such that all variables of $p$ are contained in
$\Wedge^+ E_n$, i.e., they are of the form $e_J$ with $J \supseteq
\{n+1,n+2,\ldots\}$.

Now act on $p$ with the element $f_{i_1} \wedge f_{i_2} \in
\so(V_\infty)$ with $i_1<i_2$ the two smallest elements in $I$. Since $X$ is $\Spin(V_\infty)$-stable, the result $p_1$ is again in the ideal of $X$. Furthermore, $p_1$ has the form
\[ 
    p_1=\pm e_{I \setminus \{i_1,i_2\}} \cdot q + r_1 
\]
where $q=\frac{\partial p}{\partial e_{I}}$ contains only variables $e_J$ with $|J^c| \leq k$ and where $r_1$ does not contain $e_{I \setminus \{i_1,i_2\}}$ but may contain other variables $e_J$ with $|J^c|=k+2$ (namely, those with $i_1,i_2 \not \in J$ for which $e_{J \cup \{i_1,i_2\}}$ appears in $p$).

If $n=k+2$, then $I \setminus \{i_1,i_2\}=\{n+1,n+2,\ldots\}$ and, since all variables $e_J$ in $p_1$ satisfy $J \supseteq  \{n+1,n+2,\ldots \}$, $e_{I \setminus \{i_1,i_2\}}$ is the only variable $e_J$ in $p_1$ with $|J^c|=k+2$. If $n>k+2$, then we continue in the same manner, now acting with $f_{i_3} \wedge f_{i_4}$ on $p_1$, where $i_3<i_4$ are the two smallest elements in $I \setminus \{i_1,i_2\}$. We write $p_2$ for the result, which is now of the form 
\[
    p_2=\pm e_{I \setminus \{i_1,i_2,i_3,i_4\}} \cdot q + r_2 
\]
where $q$ is the same polynomial as before and $r_2$ does not
contain the variable $e_{I \setminus \{i_1,i_2,i_3,i_4\}}$ but may contain other variables $e_J$ with $|J^c|=k+4$. 

Iterating this construction we find the polynomial
\[
    p_\ell=\pm e_{\{n+1,n+2,\ldots\}} \cdot q + r_\ell 
\]
in the ideal of $X$, where $\ell=(n-k)/2$, $q$ is the same polynomial as before and $r_\ell$ only contains variables $e_J$ with $|J^c|<n$. Let $Z:=X[1/q]$ be the open subset of $X$ where $q$ is nonzero. 

\begin{lemma} \label{lm:Z}
    For every variable $e_J$ with $|J^c| \geq n$, the ideal of $Z$ in the localisation $R^+[1/q]$ contains a polynomial of the form $e_J - s/q^d$ for some $d \in \Z_{\geq 0}$ and some $s \in R^+_{\leq n-2}$.
\end{lemma}

\begin{proof}
    We proceed by induction on $|J^c|=:m$. By successively acting  on $p_\ell$ with the elements $f_n \wedge f_{n+1},f_{n+2} \wedge f_{n+3},\ldots,f_{m-1}\wedge f_m$, we find the polynomial
    \[ 
        \pm e_{\{m+1,m+2,\ldots\}} \cdot q + r 
    \]
    in the ideal of $X$, where $r$ contains only variables $e_L$ with $|L^c|<m$. Now act with elements of $\gl(E_\infty)$ to obtain an element 
    \[ 
        \pm e_J \cdot q + \tilde{r} 
    \]
    where $\tilde{r}$ still contains only variables $e_L$ with $|L^c|<m$. Inverting $q$, this can be used to express $e_J$ in such variables $e_L$. By the induction hypothesis, all those $e_L$ admit an expression, on $Z$, as a polynomial in $R^+_{\leq n-2}$ times a negative power of $q$. Then the same holds for $e_J$. 
\end{proof}

\begin{lemma} \label{lm:ZNoeth}
    The open subscheme $Z=X[1/q]$ is stable under the group $H_n$ and $H_n$-Noetherian.
\end{lemma}

\begin{proof}
    By Lemma~\ref{lm:GLE}, $X$ is stable under $\GL(E_\infty)_n$. The polynomial $q$ is homogeneous and contains only variables $e_J$ with $J \supseteq \{n+1,n+2,\ldots\}$. Every $g \in \GL(E_\infty)_n$ scales each such variable with $\det(g)$, and hence maps $q$ to a scalar multiple of itself. We conclude that $Z$ is stable under $\GL(E_\infty)_n$, hence by (a slight variant of) Lemma~\ref{lm:GLE} also under $H_n$.

    By Lemma~\ref{lm:Z}, the projection dual to the inclusion $R^+_{\leq n-2}[1/q] \subseteq R^+[1/q]$ restricts on $Z$ to a closed embedding, and this embedding is $H_n$-equivariant. By Proposition~\ref{prop:GLE}, the image of $Z$ is $H_n$-Noetherian, hence so is $Z$ itself.
\end{proof}

\begin{proof}[Proof of Theorem~\ref{thm:Main}]
    Let 
    \[ 
        X \supseteq X_1 \supseteq \ldots 
    \]
    be a chain of reduced, $\Spin(V_{\infty})$-stable closed subschemes. Let $Y \subseteq X$ be the reduced closed subscheme defined by the orbit $\Spin(V_{\infty}) \cdot q$. Since $q$ has degree $\delta_X-1$, we have $\delta_Y < \delta_X$ and hence $Y$ is $\Spin(V_{\infty})$-Noetherian by the induction hypothesis. It follows that the chain
    \[
        Y \supseteq (Y \cap X_1)^\red \supseteq \ldots 
    \]
    is eventually stable. On the other hand, the chain 
    \[ 
        Z \supseteq (Z \cap X_1)^\red \supseteq \ldots 
    \]
    consists of reduced, $H_n$-stable closed subschemes of $Z$, hence it is eventually stable by Lemma~\ref{lm:ZNoeth}.

    Now pick a
     (not necessarily closed)
    point $P \in X_i$ for $i \gg 0$. If $P \in Y \cap X_i$, then $P \in Y \cap X_{i-1}$ by the first stabilisation. On the other hand, if $P \not \in Y \cap X_i$, then there exists a $g \in \Spin(V_\infty)$ such that $gP \in Z$. Then $gP$ lies in $X_i \cap Z$, which by the second stabilisation equals $X_{i-1} \cap Z$, hence $P=g^{-1}(gP)$ lies in $X_{i-1}$, as well. We conclude that the chain $(X_i)_i$ of closed, reduced subschemes of $X$ stabilises. Hence the inverse half-spin representation $(\Wedge^+_\infty E_\infty)^*$ is topologically $\Spin(V_\infty)$-Noetherian.
\end{proof}

\begin{remark}\label{rem:so-noetherianity}
    While the proof of Theorem~\ref{thm:Main} for the even half-spin case is easily adapted to a proof for the odd half-spin case, we \textit{do not know whether the spin representation $(\Wedge_\infty E_\infty)^*$ itself is topologically $\Spin(V_\infty)$-Noetherian!} Also, despite much effort, we have not succeeded in proving that the inverse limit $\varprojlim_n \Wedge^n V_n$ along the contraction maps $c_{e_n}$ is topologically $\SO(V_\infty)$-Noetherian. Indeed, the situation is worse for this question: like the inverse spin representation, this limit is the dual of a countable-dimensional module that splits as a direct sum of two $\SO(V_\infty)$-modules---and here we do not even know whether the dual of \textit{one} of these modules is topologically Noetherian!
\end{remark}

\section{Half-spin varieties and applications}
\label{s:spin-vars}

In this section we introduce the notion of half-spin varieties and reformulate our main result \cref{thm:Main} in this language. We start by fixing the necessary data determining the half-spin representations of $\Spin(V)$.

\begin{notation}
    As shorthand, we write $\mathcal{V}=(V,q,F) \in
    \mathcal{Q}$ to refer to a triple where
    \begin{enumerate}
        \item $V$ is an even-dimensional vector space over $K$,
        \item $q$ is a nondegenerate symmetric quadratic form on $V$, and
        \item $F$ is a maximal isotropic subspace of $V$. 
    \end{enumerate}
    An \textit{isomorphism} $\cV \to \cV'=(V',q',F')$ of such triples
    is a linear bijection $\phi: V \to V'$ with $q'(\phi(v))=q(v)$
    and $\phi(F)=F'$.
\end{notation}

Given a triple $\cV$, we have half-spin representations $\Cl^{\pm}(V)f$, where $f=f_1 \cdots f_n$ with $f_1,\ldots,f_n$ a basis of $F$ (recall that the left ideal $\Cl(V)f$ does not depend on this basis). Half-spin varieties are $\Spin(V)$-invariant subvarieties of these half-spin representations that are preserved by the contraction maps $\pi_e$ from \S\ref{ssec:Contraction} and the multiplication maps $\tau_h$ from \S\ref{ssec:Multiplication}. The precise definition below is inspired by the that of a Pl\"ucker variety in \cite{DE18}. It involves a uniform choice of either even or odd half-spin representations. For convenience of notation, we will only explicitly work with the even half-spin representations, but all further results are valid for the odd counterparts as well.

\begin{definition}[Half-spin variety] \label{d:spin-variety}
    A \textit{half-spin variety} is a rule $X$ that assigns to
    each triple $\cV=(V,q,F) \in \cQ$ a closed, reduced subscheme $X(\cV) \subseteq \Cl^+(V)f$ such that
    \begin{enumerate}
        \item $X(\cV)$ is $\Spin(V)$-stable;
        \item for any isomorphism $\phi:\cV \to \cV'$, the map $\Cl^+(\phi)$ maps $X(\cV)$ into $X(\cV')$; 
        \item for any isotropic $e \in V$ with $e \not \in F$, if we set $V':=e^\perp/\lspan{e}$, $q'$ the induced form on $V'$, $F'$ the image of $F \cap e^\perp$ in $V'$, and $\cV':=(V',q',F')$, then the contraction map $\pi_e:\Cl^+(V)f \to \Cl^+(V')f'$ maps $X(\cV)$ into $X(\cV')$; and 
        \item for any $\cV=(V,q,F)$, if we construct a triple $\cV'$ by setting $V':=V \oplus \lspan{e,h}$, $q'$ the quadratic form that restricts to $q$ on $V$, makes the direct sum orthogonal, and makes $e,h$ a hyperbolic basis, if we set $f':=f \cdot h$, then the map $\tau_h: \Cl^+(V)f \to \Cl^+(V')f'$ maps $X(\cV)$ into $X(\cV')$. 
	\end{enumerate}
\end{definition}

\begin{example} The following are
examples of half-spin varieties.
    \begin{enumerate}
      \item Trivially, $X(\cV) := \Cl^+(V)f$, $X(\cV) :=
      \{0\}$ and $X(\cV) := \emptyset$ define half-spin varieties.
     \item The even component of the cone over the isotropic
     Grassmannian, $X(\cV):=\widehat{\Gr}_{\iso}^+(V,q)$ is a
     half-spin variety by \Cref{prop:properties}. 
     \item For two half-spin varieties $X$ and $X'$ their \textit{join} $X + X'$ defined by
        \[
            (X + X')(\cV) := \overline{\{x+x' \mid x \in X(\cV), x' \in X'(\cV)\}}
        \]
        is a half-spin variety.
      \item The intersection of two half-spin varieties $X$ and $X'$
      is a half-spin variety, which is defined by $(X \cap X')(\cV) :=
      X(\cV) \cap X'(\cV)$.
   \end{enumerate}
\end{example}

Similar as in \S\ref{sec:infinte half-spin representations} we will use the following notation: for every $n \in \N$, we consider the vector space $V_n = \langle e_1,\dots, e_n,f_1,\dots,f_n\rangle$ together with the quadratic form $q_n$ whose corresponding bilinear form $(\cdot|\cdot)$ satisfies 
\[
    (e_i|e_j)=0, \quad (f_i|f_j)=0 \quad \text{and} \quad (e_i|f_j)=\delta_{ij}.
\]
Furthermore, let $E_n=\langle e_1,\dots,e_n\rangle$ and $F_n=\langle
f_1,\dots,f_n\rangle$; these are maximal isotropic subspaces of $V_n$. We will denote the associated tuple by $\cV_n=(V_n,q_n,F_n)$. 

\begin{remark}\label{rem:spin variety determined by Xn} 
    A half-spin variety $X$ is completely determined by the values $X(\cV_n)$, that is, if $X$ and $X'$ are half-spin varieties such that $X(\cV_n) = X'(\cV_n)$ for all $n \in \N$, then $X(\cV) = X'(\cV)$ for all tuples $\cV$.
\end{remark}

We now want to associate to each half-spin variety $X$ an infinite-dimensional scheme $X_\infty$ embedded inside the inverse half-spin representation $(\Wedge^+_\infty E_\infty)^*$ as follows.  Since $V_n = E_n \oplus F_n$, we can use the isomorphism from \S\ref{ssec:ExplicitFormulas} to embed $X(\cV_n)$ as a reduced subscheme of $\bigwedge^+ E_n$ (recall from \S\ref{subsec:half-spin repr} that we view $\bigwedge^+ E_n$ as the affine scheme with coordinate ring ${\rm Sym}(\Wedge^{+(-1)^n}E_n)$).
We abbreviate $X_n \coloneqq X(\cV_n) \subseteq \bigwedge^+E_n$.

For $N \geq n$ let $\pi_{N,n}: \bigwedge \nolimits^+E_N \to \bigwedge \nolimits^+E_n$, resp. $\tau_{n,N}: \bigwedge \nolimits^+E_n \to \bigwedge \nolimits^+E_N$ be the maps induced by by canonical projection $E_N \to E_n$, resp.\ by the injection $E_n \hookrightarrow E_N$. Note that $\tau_{n,N}$ is a section of $\pi_{N,n}$. Recall that $(\Wedge^+_\infty E_\infty)^* = \varprojlim_n \bigwedge^+E_n$. We denote the structure maps by $\pi_{\infty,n}: (\Wedge^+_\infty E_\infty)^* \to \bigwedge \nolimits^+E_n$ and by $\tau_{n,\infty}: \bigwedge \nolimits^+E_n \to (\Wedge^+_\infty E_\infty)^*$ the inclusion maps induced by $\tau_{n,N}$.

From the definition of a half-spin variety it follows that 
\begin{equation} \label{eq:pi&tau}
    \pi_{N,n}(X_N) \subseteq X_n \quad \text{and} \quad \tau_{n,N}(X_n) \subseteq X_N.
\end{equation}
Hence the inverse limit 
\[
    X_\infty \coloneqq \underset{n}{\varprojlim}\ X_n
\]
is well-defined, and a closed, reduced, $\Spin(V_\infty)$-stable subscheme of $(\Wedge^+_\infty E_\infty)^*$. In order to see this, write $R_n \coloneqq \Sym(\bigwedge^{+(-1)^n}E_n)$ and $R_\infty \coloneqq \Sym(\Wedge^+_\infty E_\infty)$. Let $I_n \subseteq R_n$ be the radical ideal associated to $X_n \subseteq \Spec(R_n)$, i.e. $X_n = V(I_n) = \Spec(R_n/I_n)$. As $\Spec(\cdot)$ is a contravariant equivalence of categories, it holds that
\[
    X_\infty \coloneqq \varprojlim_n X_n = \varprojlim_n \Spec(R_n/I_n) = \Spec(\varinjlim_n(R_n/I_n)).
\]
So $X_\infty$ corresponds to the ideal $I_\infty \coloneqq \varinjlim_n I_n \subseteq R_\infty$. As all $I_n \subseteq R_n$ are radical, so is $I_\infty \subseteq R_\infty$ and therefore $X_\infty$ is a reduced subscheme. 

It follows from \Cref{eq:pi&tau} that 
\begin{equation} \label{eq:pi&tau_inf}
     \pi_{\infty,n}(X_\infty) \subseteq X_n \quad \text{and} \quad \tau_{n,\infty}(X_n) \subseteq X_\infty.   
\end{equation}

\begin{lemma} \label{lemma:injective map}
    The mapping 
    \[
        X \mapsto X_\infty
    \]
    is injective. That is, if $X$ and $X'$ are half-spin varieties such that $X_\infty = X'_\infty$, then $X = X'$, i.e. $X(\cV) = X'(\cV)$ for all tuples $\cV$.
\end{lemma}

\begin{proof} 
    Note that, for all $n \in \N$, it holds that
    \[
        X_n = \pi_{\infty,n}(X_\infty).
    \]
    Indeed, the inclusion $\supseteq$ is contained in \Cref{eq:pi&tau_inf}, and the other direction $\subseteq$ follows from the fact that $\tau_{n,\infty}:X_n \to X_\infty$ is a section of $\pi_{\infty,n}$. 
    Hence, if $X_\infty = X'_\infty$, then
    \[
        X_n = \pi_{\infty,n}(X_\infty)=\pi_{\infty,n}(X'_\infty)=X'_n.
    \]
    By \Cref{rem:spin variety determined by Xn} this shows that $X = X'$.
\end{proof}

For two half-spin varieties $X$ and $X'$, we will write $X \subseteq
X'$  if $X(\cV) \subseteq X'(\cV)$ for all $\cV = (V,q,F)$.
\cref{thm:Main} then implies the following.

\begin{theorem}[Noetherianity of half-spin varieties] \label{t:spin-var-top-noeth} 
    Every descending chain of half-spin varieties 
    \[
        X^{(0)} \supseteq X^{(1)} \supseteq X^{(2)} \supseteq X^{(3)} \supseteq \ldots
    \]
    stabilizes, that is, there exists $m_0 \in \N$ such that $X^{(m)} = X^{(m_0)}$ for all $m \geq m_0$. 
\end{theorem}

\begin{proof}
    Note that the mapping $X \mapsto X_\infty$ is order preserving, that is, if $X \subseteq X'$, then $X_\infty \subseteq X'_\infty$. Hence, a chain 
    \[
      X^{(0)} \supseteq X^{(1)} \supseteq X^{(2)} \supseteq X^{(3)} \supseteq \ldots
    \]
    of half-spin varieties induces a chain 
    \[
       X^{(0)}_\infty \supseteq X^{(1)}_\infty \supseteq X^{(2)}_\infty \supseteq X^{(3)}_\infty \supseteq \ldots
    \]
    of closed, reduced, $\Spin(V_\infty)$-stable subschemes in $(\Wedge^+_\infty E_\infty)^*$. By \Cref{thm:Main} we know that $(\Wedge^+_\infty E_\infty)^*$ is topologically $\Spin(V_\infty)$-Noetherian. Hence, the chain $X^{(m)}_\infty$ stabilizes. But then, by \Cref{lemma:injective map} also the chain of half-spin varieties $X^{(m)}$ stabilizes. This completes the proof.
\end{proof}

As a consequence we obtain the next results, which state how $X_\infty$ is determined by the data coming from some finite level of $X$. 

\begin{theorem}
    
\label{thm:infinite_to_finite_spinvariety}
    Let $X$ be a half-spin variety. Then there exists $n_0 \in \N$ such that 
    \[
        X_\infty = V\big(\rad( \Spin(V_\infty) \cdot I_{n_0} )\big),
    \]
    where $\rad( \Spin(V_\infty) \cdot I_{n_0} ) \subseteq \Sym(\Wedge^+_\infty E_\infty)$ is the radical ideal generated by the $\Spin(V_\infty)$-orbits of the ideal $I_{n_0} \subseteq \Sym(\bigwedge^{+(-1)^{n_0}}E_{n_0})$ defining $X_{n_0} \subseteq \bigwedge^+E_{n_0}$.
\end{theorem} 

\begin{proof}
    For each $n \in \N$ set $J_n \coloneqq \rad(\Spin(V_\infty) \cdot I_{n}) \subseteq \Sym(\Wedge^+_\infty E_\infty)$. We denote by $I_\infty \subseteq \Sym(\Wedge^+_\infty E_\infty)$ the ideal associated to $X_\infty$. This ideal is $\Spin(V_\infty)$-stable, radical and it holds that $I_\infty = \varinjlim_n I_n$. Thus, $\bigcup_n J_n = I_\infty$. Since $(J_n)_{n\in \N}$ is an increasing chain of closed $\Spin(V_\infty)$-stable radical ideals, by \Cref{thm:Main} there exists $n_0 \in \N$ such that $J_n=J_{n_0}$ for all $n\geq n_0$. Therefore, $I_\infty = \bigcup_n J_n = J_{n_0}$ and hence $X_\infty=V(I_\infty)=V(J_{n_0}).$ 
\end{proof}

\begin{corollary} \label{cor:finite_to_finite_spinvariety}
    Let $X$ be a half-spin variety. There exists $n_0 \in \N$ such that for all $n\geq n_0$ it holds that
    \[
        X_n = V(\rad(\Spin(V_n)\cdot I_{n_0})).
    \]
\end{corollary}

\begin{proof}
Take $n_0$ as in Theorem~\ref{thm:infinite_to_finite_spinvariety}.
Then the statement follows from that theorem and \cite[Lemma
2.1]{Dra10}. To apply that lemma, we must check condition (*) in that
paper, namely, that for $q \geq n \geq n_0$ and $g \in \Spin(V_q)$ we can write 
\[ \pi_{q,n_0} \circ g \circ \tau_{n,q} = g'' \circ \tau_{m,n_0}
\circ \pi_{n,m} \circ g' \]
for suitable $m \leq n_0$ and $g' \in \Spin(V_n)$ and $g'' \in
\Spin(V_{n_0})$. In fact, since half-spin varieties are affine cones, it
suffices that this identity holds up to a scalar factor. It also suffices
to prove this for $g$ in an open dense subset $U$ of $\Spin(V_q)$,
because the equations for $X_{n_0}$ pulled back along the map on the left
for $g \in U$ imply the equations for all $g$.  We will prove this,
with $m=n_0$, using the Cartan map in Lemma~\ref{lm:Lower} below.
\end{proof}

\section{Universality of \texorpdfstring{$\^\Gr^+_{\iso}(4,8)$}{the isotropic Grassmannian in dimension 8} and the Cartan map}\label{sec:Cartan}

\subsection{Statement} \label{ssec:universality grassmannian}

In \cite{ST21} the last two authors showed that in even dimension,
the isotropic Grassmannian in its Pl\"ucker embedding is
set-theoretically defined by pulling back equations coming from
$\widehat{\Gr}_{\iso}(4,8)$. Using the \textit{Cartan map} we can
translate this into a statement about the isotropic Grassmannian in its
spinor embedding and prove the following result.

\begin{theorem}\label{cor:isotropic-cartan}
    For all $n\geq 4$ we have 
    \[
        \widehat{\Gr}^+_{\iso}(V_n) = V(\rad(\Spin(V_n)\cdot I_4)), 
    \]
    where $I_4$ is the ideal of polynomials defining
    $\widehat{\Gr}^+_{\iso}(V_4) \subseteq \Cl^+(V_4)f.$
\end{theorem}

In other words, the bound $n_0$ from
Corollary~\ref{cor:finite_to_finite_spinvariety} can be taken equal to
$4$ for the cone over the isotropic Grassmannian.  We give the proof of
\Cref{cor:isotropic-cartan} in \S\ref{subsec:lastproof} using properties
of the Cartan map that will be established in the following sections.

\subsection{Definition of the Cartan map} \label{ssec:Exterior}

When we regard $e_1 \wedge \cdots \wedge e_n$ as an element of the
$n$-th exterior power $\Wedge^n V$ of the standard representation
$V$ of $\so(V)$, then it is a highest weight vector of weight
$(0,\ldots,0,2)=2\lambda_0$, where $\lambda_0$ is the fundamental
weight introduced in \S\ref{ssec:highest weight} and the highest
weight of the half-spin representation $\Cl^{(-1)^n}(V) f$. Similarly, the
element $e_1 \wedge e_2 \wedge \cdots \wedge e_{n-1} \wedge f_n \in
\Wedge^n V$ is a highest weight vector of weight $(0,\ldots,0,2,0)=2
\lambda_1$, where $\lambda_1$ is the highest weight of the other
half-spin representation. So $\Wedge^n V$ contains copies of the irreducible
representations $V_{2\lambda_0},V_{2 \lambda_1}$ of $\so(V)$; in fact,
it is well known to be the direct sum of these. To compare our results in
this paper about spin representations with earlier work by the last two
authors about exterior powers, we will need the following considerations.

Consider any connected, reductive algebraic group $G$, with maximal torus $T$ and Borel subgroup $B \supseteq T$. Let $\lambda$ be a dominant weight of $G$, let $V_{\lambda}$ be the corresponding irreducible representation, and let $v_\lambda \in V_{\lambda}$ be a nonzero highest-weight vector (which is unique up to scalar multiples). Then the symmetric square $S^2 V_\lambda$ contains a one-dimensional space of vectors of weight $2 \lambda$, spanned by $v_{2\lambda}:=v_\lambda^2$. This vector is itself a highest-weight vector, and hence generates a copy of $V_{2\lambda}$; this is sometimes called the \textit{Cartan component} of $S^2 V_\lambda$. By semisimplicity, there is a $G$-equivariant linear projection $\pi: S^2 V_\lambda \to V_{2\lambda}$ that restricts to the identity on $V_{2\lambda}$. The map
\[ 
    \widehat{\nu_2}: V_\lambda \to V_{2\lambda},\quad  v \mapsto \pi(v^2).
\]
is a nonzero polynomial map, homogeneous of degree $2$, and hence induces a rational map $\nu_2:\PP V_{\lambda} \to \PP V_{2\lambda}$. Note that this is the composition of the quadratic Veronese embedding and the projection $\pi$. We will refer to $\nu_2$ and to $\^{\nu_2}$ as the \textit{Cartan map}. 

\begin{lemma} \label{lm:CartanInjective}
The rational map $\nu_2$ is a morphism and injective.
\end{lemma}

We thank J.M.~Landsberg for help with the following proof.

\begin{proof}
    To show that $\nu_2$ is a morphism, we need to show that $\pi(v^2)$ is nonzero whenever $v$ is. Now the set $Q$ of all $[v] \in \PP V_\lambda$ for which $\pi(v^2)$ is zero is closed and $B$-stable. Hence, if $Q \neq \emptyset$, then by Borel's fixed point theorem, $Q$ contains a $B$-fixed point. But the only $B$-fixed point in $\PP V_\lambda$ is $[v_\lambda]$, and $v_\lambda$ is mapped to the nonzero vector $v_{2\lambda}$. Hence $Q=\emptyset$.

    Injectivity is similar but slightly more subtle. Assume that there exist distinct $[v],[w]$ with $\nu_2([v])=\nu_2([w])$. Then $\{[v],[w]\}$ represents a point in the Hilbert scheme of two points in $\PP V_\lambda$. Now the locus $Q$ of points $S$ in said Hilbert scheme such that $\nu_2(S)$ is a single reduced point is a closed subset of a projective scheme, hence $Q$ contains a $B$-stable point $S$. This scheme $S$ cannot consist of two distinct reduced points, since there is only one $B$-stable point. Therefore, the reduced subscheme of $S$ is $\{[v_{\lambda}]\}$, and $S$ represents the point $[v_\lambda]$ together with a nonzero tangent direction in $T_{[v_\lambda]} \PP V_{\lambda} = V_\lambda/K v_\lambda$, represented by $w \in V_\lambda$. Furthermore, $B$-stability of $S$ implies that the $B$-module generated by $w$ equals $\langle w,v_\lambda \rangle_K$. That $S$ lies in $Q$ means that
    \[
        \pi((v_\lambda+\epsilon w)^{2})=v_{2\lambda}  \mod \epsilon^2.
    \]
    We find that $\pi(v_\lambda w)=0$, so that the $G$-module generated by $v_\lambda w \in S^2 V$ does not contain $V_{2\lambda}$. But since $v_\lambda $ is (up to a scalar) fixed by $B$, the $B$-module generated by $v_\lambda w$ equals $v_\lambda $ times the $B$-module gene rated by $w$, and hence contains $v_\lambda^2=v_{2\lambda}$, a contradiction.
\end{proof}

Observe that $\nu_2$ maps the unique closed orbit $G \cdot [v_\lambda]$ in $\PP V_\lambda$ isomorphically to the unique closed orbit $G \cdot [v_{2\lambda}]$---both are isomorphic to $G/P$, where $P \supseteq B$ is the stabiliser of the line $K v_\lambda$ and of the line $K v_{2\lambda}$. In our setting above, where $G=\Spin(V)$ and $\lambda \in \{\lambda_0,\lambda_1\}$, the closed orbit $G \cdot [v_{2\lambda}]$ is one of the two connected components of the Grassmannian $\Gr_{\iso}(V)$ of $n$-dimensional isotropic subspaces of $V$, in its Pl\"ucker embedding; and the closed orbit in the projectivised half-spin representation $\PP V_{\lambda}$ is the same component of the isotropic Grassmannian but now in its spinor embedding.

In what follows we will need a more explicit understanding both of the embedding of the isotropic Grassmannian in the projectivised (half-)spin representations and of the map $\^{\nu_2}$. These are treated in the next two paragraphs.

\subsection{The map \texorpdfstring{$\^{\nu_2}$}{} from the spin representation to the exterior power} \label{ssec:Exterior2}

In \S\ref{ssec:Exterior} we argued the existence of $\Spin(V)$-equivariant quadratic maps from the half-spin representations to the two summands of $\Wedge^n V$. In \cite{manivel2009spinor} these two maps are described jointly as 
\[ 
    \^{\nu_2}:\Cl(V)f \to \Wedge^n V, \quad af \mapsto \text{ the component in $\Wedge^n V$ of } (afa^*) \bullet 1 \in \Wedge V, 
\]
where $\bullet$ stands for the $\Cl(V)$-module structure of $\Wedge V$ from \S\ref{ssec:Grassmann} and $a^*$ refers to the anti-automorphism of the Clifford algebra from \S\ref{ssec:Clifford}.

\begin{lemma}\label{ex:Cartan}
     The map $\hat{\nu}_2$ maps the isotropic Grassmann cone in its
     spinor embedding to the isotropic Grassmann cone in its Pl\"ucker embedding, i.e., 
     \[
        \hat{\nu}_2\big( \widehat{\Gr}_{\iso}(V)\big) = \widehat{\Gr}^{\operatorname{Pl}}_{\iso}(V),
    \]
    where $\widehat{\Gr}^{\operatorname{Pl}}_{\iso}(V)$ is the isotropic Grassmann cone in its Pl\"ucker embedding (see \cite[Definition 3.7]{ST21}).
\end{lemma}

\begin{proof}
    Let $H \subseteq V$ be a maximal isotropic subspace that
    intersects $F$ in a $k$-dimensional space. Choose a hyperbolic
    basis $e_1,\ldots,e_n,f_1,\ldots,f_n$ adapted to $H$ and $F$, so
    that $H=\lspan{e_{k+1},\ldots,e_n,f_1,\ldots,f_k}$ is represented
    by the vector $\omega_H:= e_{k+1} \cdots e_n f 
    \in \widehat{\Gr}_{\iso}(V)$ 
    where $f=f_1 \cdots f_n$; see \S\ref{ssec:igc}. Set $a:=e_{k+1}
    \cdots e_n$. 
    Now
    \begin{align*} 
        afa^*&=e_{k+1} \cdots e_n f_1 \cdots f_n e_n \cdots e_{k+1}\\
        &=e_{k+1} \cdots e_n f_1 \cdots f_{n-1} (2-e_n f_n) e_{n-1} \cdots e_{k+1} \\
        &=2 e_{k+1} \cdots e_n f_1 \cdots f_{n-1} e_{n-1} \cdots e_{k+1} \\
        &=\ldots\\
        &=2^{n-k} e_{k+1} \cdots e_n f_1 \cdots f_k 
    \end{align*}
    where we have used the definition of $\Cl(V)$ (in the first step), the fact that the second copy of $e_n$ is perpendicular to all elements before it and multiplies to zero with the first copy of $e_n$ (in the second step), and have repeated this another $n-k-1$ times in the last step. We now find that 
    \[ 
        (afa^*) \bullet 1=2^{n-k} e_{k+1} \wedge \cdots \wedge e_n \wedge f_1 \wedge \cdots \wedge f_k, 
    \]
    so that $(afa^*) \bullet 1$ lies in one of the two summands of
    $\Wedge^n V$ and spans the line representing the space $H$ in the
    Pl\"ucker embedding. This shows that $\hat{\nu_2}$ maps the
    isotropic Grassmann cone in its spinor embedding to the isotropic
    Grassmann cone in its Pl\"ucker embedding, as desired.
\end{proof}

\begin{remark}
    While the spin representation $\Cl(V)f$ depends only on the space
    $F$---since $F$ determines $f$ up to a scalar, which doesn't
    affect the left ideal $\Cl(V)f$---the map $\^{\nu_2}$ actually
    depends on $f$ itself: for $\tilde{f}:=t f$ with $t \in K^*$, the
    map $\^{\nu_2}$ constructed from $\tilde{f}$ sends $af= (t^{-1}a)\tilde{f}$ to $t^{-1}a \tilde{f} t^{-1}a^*=t^{-1}a f a^*$, so the new $\^{\nu_2}$ is $t^{-1}$ times the old map. 
\end{remark}

\subsection{Contraction and the Cartan map commute}
Recall from \S\ref{ssec:Exterior} that we have quadratic maps $\^{\nu_2}$ from the half-spin representations to the two summands of $\Wedge^n V$; together, these form a quadratic map $\^{\nu_2}$ which we discussed in \S\ref{ssec:Exterior2}. By abuse of terminology, we call this, too, the Cartan map. Given an isotropic vector $e \in V$ that is not in $F$, we write $\^{\nu_2}$ also for the Cartan map $\Cl(V_e)\ol{f} \to \Wedge^{n-1} V_e$ (notation as in \S\ref{ssec:Contraction}). Recall from \S\ref{ssec:Contraction} the contraction map $c_e: \Wedge^n V \to \Wedge^{n-1} V_e$ and its spin analogue $\pi_e: \Cl(V)f \to \Cl(V_e)\ol{f}$. Also, for a fixed $h=f_n \in F$ with $\langle e,h \rangle=1$, recall from \S\ref{ssec:Multiplication} the multiplication map
$m_h:\Wedge^{n-1} V_e \to \Wedge^n V$ and its spin analogue
$\tau_h:\Cl(V_e) \ol{f} \to \Cl(V) f$. The relations between these
maps are as follows. 

\begin{proposition} \label{prop:CartanContraction}
    The following diagrams essentially commute:
    \begin{equation} \label{eq:FinitePluecker}
        \xymatrix{
        \Cl(V)f \ar[d]_{\^{\nu_2}} \ar[r]^-{\pi_e} & \Cl(V_e) \ol{f} \ar[d]^{\^{\nu_2}} \\
        \Wedge^n V \ar[r]_-{c_e} & \Wedge^{n-1} V_e}
	\text{ and }
	\xymatrix{
	\Cl(V_e)\ol{f} \ar[d]_{\^{\nu_2}} \ar[r]^-{\tau_h} & \Cl(V)f \ar[d]^{\^{\nu_2}} \\
	\Wedge^{n-1} V_e \ar[r]_-{m_h} & \Wedge^n V.
	}
    \end{equation}
    More precisely, one can rescale the restrictions of $c_e$ to the
    two $\so(V)$-submodules of $\Wedge^n V$ each by $\pm 1$ in such a
    manner that the diagram commutes, and similarly for $m_h$. 
\end{proposition}

Naturally, we could have chosen the scalars in the definition of $c_e$
(or, using a square root of $-1$, in that of $\pi_e$) such that the
diagram literally commutes. However, we have chosen the scalars such
that $c_e$ has the most natural formula and $\pi_e,\tau_h$ have the
most natural formulas in our model $\Wedge E$ for the spin representation.

\begin{proof}
    We may choose a hyperbolic basis $e_1,\ldots,e_n,f_1,\ldots,f_n$ of $V$ such that $e=e_n$ and $f_1,\ldots,f_n$ is a basis of $F$. We write $f:=f_1 \cdots f_n$ and $\ol{f}:=\ol{f}_1 \cdots \ol{f}_{n-1}$.

    Since the vertical maps are quadratic, it is not sufficient to show commutativity on a spanning set. We therefore consider 
    \[ 
       a:=\sum_{I \subseteq [n]} c_I e_I 
    \]
    where, for $I=\{i_1<\ldots<i_k\}$ we write $e_I:=e_{i_1} \cdots e_{i_k}$. We then have
    \[
        \pi_e(af)=\sum_{I: n \not \in I} c_I \ol{e_I}\ol{f}=:\ol{a}\ol{f} 
    \]
    and 
    \[ 
        \^{\nu_2}(\ol{a}\ol{f})=\text{the component in $\Wedge^{n-1} V_e$ of} \sum_{I,J: n \not \in I \cup J} (c_I c_J \ol{e_I} \ol{f} \ol{e_J}^*) \bullet 1 \in \Wedge V_e.
    \] 
    Now note that, since $\ol{f}$ has $n-1$ factors, if $I,J$ do not have the same parity, then acting with $\ol{e_I} \ol{f} \ol{e_J}^*$ on $1$ yields a zero contribution in $\Wedge^{n-1} V_e$. Hence the sum above may be split into two sums, one of which is
    \begin{equation}\label{eq:OneWay}
        \text{the component in $\Wedge^{n-1} V_e$ of} \sum_{I,J:
	|I|,|J| \text{ even, } n \not \in I \cup J} (c_I c_J \ol{e_I} \ol{f} \ol{e_J}^*) \bullet 1.
    \end{equation}

    On the other hand, consider
    \[ 
        \^{\nu_2}(af)=\text{the component in $\Wedge^n V$ of} \sum_{I,J} (c_I c_J e_I f e_J^*) \bullet 1 \in \Wedge V.
    \]
    For the same reason as above, this splits into two sums, and we want to compare the following expression to \eqref{eq:OneWay}:
    \begin{equation} \label{eq:OtherWay}
        c_e(\text{the component in $\Wedge^n V$ of } \sum_{I,J:|I|,|J|
	\text{ even}} (c_I c_J e_I f e_J^*) \bullet 1).
    \end{equation}
    Now recall that the action of $e=e_n \in V \subseteq \Cl(V)$ on $\Wedge V$ is via $o(e)+\iota(e)$, while $c_e$ is $\iota_e$ followed by projection to $\Wedge^{n-1} V_e$. Hence to compute \eqref{eq:OtherWay}, we may as well compute the summands of 
    \[ 
        \text{the component in $\Wedge^n V$ of }
	\sum_{I,J:|I|,|J|\text{ even}} (c_I c_J \cdot e \cdot e_I f e_J^*) \bullet 1 
    \]
    that do not contain a factor $e$. Terms with $n \in I$ do not contribute, because then $e e_I=0$. Terms with $n \not \in I$ but $n \in J$ do not contribute because when $e$ gets contracted with $f_n$ a factor $e$ in $e_J^*$ survives, and when $e$ does not get contracted with $f_n$, we use $ee_J^*=0$. So we may restrict attention to the terms with $n \not \in I \cup J$. Let $I,J$ correspond to such a term, that is, $|I|,|J|$ are even and $n \not \in I \cup J$. Write $I=\{i_1<\ldots<i_k\}$ and $J=\{j_1<\ldots<j_l\}$. Then 
    \begin{align*} 
        (e e_I f e_J^*) \bullet 1
        &=((-1)^{n-1} e_I f_1 \cdots f_{n-1} e f_n e_J^*) \bullet 1 \\
        &=((-1)^{n-1} e_I f_1 \cdots f_{n-1} e) \bullet (f_n \wedge e_{j_l} \wedge \cdots \wedge e_{j_1}) \\
        &=((-1)^{n-1} e_I f_1 \cdots f_{n-1}) \bullet (e_{j_l} \wedge \cdots \wedge e_{j_1} + e \wedge f_n \wedge e_{j_l} \wedge \cdots \wedge e_{j_1}). 
    \end{align*}
    The second term in the last expression will contribute only terms with a factor $e$ to the final result, and the former term contributes
    \[
        \text{the component in $\Wedge^{n-1} V_e$ of } (-1)^{n-1} (\ol{e_I} \ol{f} \ol{e_J}^*) \bullet 1. 
    \]
    Comparing this with \eqref{eq:OneWay}, we see that the diagram commutes on terms in $\Cl^+(V)f$ up to the factor $(-1)^{n-1}$. A similar computation shows that it commutes on terms in $\Cl^-(V)f$ up to a factor factor $(-1)^n$.

We now consider the second diagram, where $V$ is split as the orthogonal
direct sum $V_e \oplus \langle e,h \rangle$ with $e=e_n,h=f_n$. Consider
$a \in \Cl(\lspan{e_1,\ldots,e_{n-1}})$. By the same argument as above, it suffices to consider the
case where all summands of $a$ in the basis $e_I$ have indices $I$ with
$|I|$ of the same parity, say even. Then $\^\nu_2 \circ \tau_h$ in the diagram sends $a\ol{f}$ to the component in $\Wedge^n V$ of $afa^*
\bullet 1$. Since the summands $e_I$ in $a$ all have $n \not \in I$,
in $a f a^* \bullet 1$ all summands have a factor $f_n$, and indeed
\[ (a f a^*) \bullet 1 = f_n \wedge (a \ol{f} a^* \bullet 1) \]
(when all terms in $a$ have $|I|$ odd, we get a minus sign). The
component in $\Wedge^n V$ of this expression is the same as the one obtained via $m_h \circ \^\nu_2$.
\end{proof}

\subsection{Proof of \texorpdfstring{\Cref{cor:isotropic-cartan}}{Theorem 6.1}}
\label{subsec:lastproof}

In this section we use the Cartan map to prove \cref{cor:isotropic-cartan}, and finish the proof of \cref{cor:finite_to_finite_spinvariety} via a similar argument.

\begin{proof}[Proof of Theorem~\ref{cor:isotropic-cartan}] 
    For a quadratic space of dimension $2n$, denote by $\widehat{\Gr}^{\operatorname{Pl}}_{\iso}(V) \subseteq \Wedge^n V$ the isotropic Grassmann cone over the Pl\"ucker embedding. Given a maximal isotropic subspace $F \subseteq V$ with basis $f_1,\ldots,f_n$ and $f:=f_1 \cdots f_n$, let $\hat{\nu}_2:\Cl^+(V)f \to \Wedge^n V$ be the Cartan map defined in \S\ref{ssec:Exterior2}. For any isotropic $v \in V \setminus F$ the diagram
    \[
    \begin{tikzcd}
    \Wedge^nV \arrow{r}{c_v}  & \Wedge^{n-1}V_v \\
    \Cl(V)f \arrow{r}{\pi_v} \arrow{u}{\hat{\nu}_2} & \Cl(V_v)\ol{f} \arrow[swap]{u}{\hat{\nu}_2}
  \end{tikzcd}
    \]
    commutes up to scalar factor at the bottom by \Cref{prop:CartanContraction}, where $V_v:=v^\perp/\lspan{v}$ and where $\ol{f}$ is the image of a product of a basis of $v^\perp \cap F_n$.

    The proof of \cite[Corollary 4.2]{ST21} shows that for $\omega \in \Wedge^n V$ the following are equivalent:
    \begin{enumerate}
        \item $\omega \in \widehat{\Gr}^{\rm Pl}_{\iso}(V)$;
        \item For every sequence of isotropic vectors $v_1 \in V$, $v_2 \in V_{v_1}$, $v_3 \in (V_{v_1})_{v_2}, \dots, v_{n-4} \in (\cdots((V_{v_1})_{v_2})_{v_3}\cdots)_{v_{n-3}}$ it holds 
        \[
            C(\omega) \in \widehat{\Gr}^{\rm Pl}_{\iso}(W),
        \]
        where we abbreviate
        $W:=(\cdots((V_{v_1})_{v_2})_{v_3}\cdots)_{v_{n-4}}$ and
        $C:\Wedge^n V \to \Wedge^{4}W$ is the composition $C:= c_{v_{n-4}}
        \circ \cdots \circ c_{v_1}$ of the contraction maps $c_{v_i}$
        introduced in \Cref{ssec:Contraction}. 
    \end{enumerate}
    By slight abuse of notation, we also write $v_1,\ldots,v_{n-4}$ for preimages of these vectors in $V$. These span an $(n-4)$-dimensional isotropic subspace $U$ of $V$ (provided that each $v_i$ chosen above in the successive quotients is nonzero), and $W$ equals $U^\perp / U$. For any fixed $\omega$, the condition that $C(\omega)$ lies in $\widehat{\Gr}^{\rm Pl}_{\iso}(W)$ is a closed condition on $U$, and hence it suffices to check that condition for $U$ in a dense subset of the Grassmannian of isotropic $(n-4)$-dimensional subspaces of $V$. In particular, it suffices to check this when $U \cap F_n=\{0\}$.

    Fix $n \geq 4$ and $x \in \Cl(V_n)f_1 \cdots f_n$ such that $p(g \cdot x)=0$ for all $g \in {\Spin}(V_n)$ and all $p \in I_4$. This means precisely that $\pi_{n,4}(g\cdot x) \in \widehat{\Gr}^{+}_{\iso}(V_4)$ for all $g \in {\Spin}(V_n)$. We need to show that $x \in \widehat{\Gr}^{+}_{\iso}(V_n)$. To this end, consider $\omega:=\hat{\nu}_2(x) \in \Wedge^nV_n$. It suffices to show that $\omega \in \widehat{\Gr}^{\operatorname{Pl}}_{\iso}(V_n)$. Indeed, this follows from the fact that $\hat{\nu}_2\big( \widehat{\Gr}^{+}_{\iso}(V)\big)$ is one of the two irreducible components of $\widehat{\Gr}^{\operatorname{Pl}}_{\iso}(V)$ (see \Cref{ex:Cartan}) and because $\nu_2$ is an injective morphism by \Cref{lm:CartanInjective}. Let $v_1,v_2,\ldots,v_{n-4} \in V_n$ as above: linearly independent, and such that the span $U:=\lspan{v_1,\ldots,v_{n-4}}$ is an isotropic space that intersects $F_n$ trivially. Let $C:= c_{v_{n-4}} \circ \cdots \circ c_{v_1}$ be the composition of the associated contractions. We need to show that $ C(\omega) \in \widehat{\Gr}^{\operatorname{Pl}}_{\iso}(W)$, where $W:=U^\perp/U$.

    Now $\hat{\nu}_2\big( \widehat{\Gr}^{+}_{\iso}(W)\big) \subseteq \widehat{\Gr}^{\rm Pl}_{\iso}(W)$ by \Cref{ex:Cartan}, and the diagram
    \[
        \begin{tikzcd}[column sep=large]
        \Wedge^nV_n \arrow{rr}{C}  & &\Wedge^{4}W \\
        \Cl(V_n) f \arrow{rr}{\pi_{v_{n-4}} \circ \cdots \circ \pi_{v_1}}
        \arrow{u}{\hat{\nu}_2} & & \
        \Cl(W) \ol{f}, \arrow[swap]{u}{\hat{\nu}_2}
    \end{tikzcd}
    \]
    where $\ol{f}$ is the image of the product of a basis of $U^\perp \cap F_n$, commutes up to a scalar factor in the bottom map due to \Cref{prop:CartanContraction}. Hence it suffices to check that $\pi_{v_{n-4}} \circ \cdots \circ \pi_{v_1}(x) \in \widehat{\Gr}^{+}_{\iso}(W)$. Now there exists an element $g \in \Spin(V_n)$ that maps $F_n$ into itself (not with the identity!) and sends $v_i$ to $e_{n+1-i}$ for $i=1,\ldots,n-4$. This induces an isometry $W:=U^\perp / U \to (U')^\perp / U'= V_4 = \lspan{ e_1,\ldots,e_4,f_1,\ldots,f_4}$, where $U':=\lspan{e_5,\ldots,e_n}$. This in turn induces a linear isomorphism (unique up to a scalar) $\Cl(W)\cdot\ol{f} \to \Cl(V_4)\cdot f_1 \cdots f_4$ (where $f$ on the left is the product of a basis of $F_n \cap U^\perp$) that maps $\widehat{\Gr}^+_{\iso}(W)$ onto $\widehat{\Gr}^+_{\iso}(V_4)$. Since, by assumption, $\pi_{n,4}(g \cdot x)=\pi_{e_5} \circ \cdots \circ \pi_{e_n}(g \cdot x)$ lies in the latter isotropic Grassmann cone, $\pi_{v_{n-4}} \circ \cdots \circ \pi_{v_1}(x)$ lies in the former. 
\end{proof}
 
\begin{lemma} \label{lm:Lower}
Let $q \geq n \geq n_0$. Then for all $g$ in some open dense subset of
$\Spin(V_q)$ there exist $g' \in \Spin(V_n)$ and $g''
\in \Spin(V_{n_0})$ such that
\[ \pi_{q,n_0} \circ g \circ \tau_{n,q} = g'' \circ 
\pi_{n,n_0} \circ g' \]
holds up to a scalar factor.
\end{lemma}

\begin{proof}
The proof is similar to that above; we just give a sketch. Using the
Cartan map, which is equivariant for the relevant spin groups, this
lemma follows from a similar statement for the corresponding (halfs of)
exterior power representations. Specifically, define
\begin{align*}
E&:=\langle e_{n_0+1},\ldots,e_q \rangle \subseteq V_q, \\
E'&:=\langle e_{n_0+1},\ldots,e_n \rangle \subseteq V_n, \text{ and} \\
F&:=\langle f_{n+1},\ldots,f_q \rangle  \subseteq V_q.
\end{align*}
Then the desired identity is
\begin{equation} \label{eq:cgm}
c_E \circ g \circ m_F = g'' \circ c_{E'} \circ g' 
\end{equation}
(up to a scalar), where 
\begin{align*}
c_E&:=c_{e_{n_0+1}} \circ \cdots \circ c_{e_q}:\Wedge^q V_q \to 
\Wedge^{n_0} V_{n_0}, \\
c_{E'}&:=c_{n_0+1} \circ \cdots \circ c_{e_n}:\Wedge^n V_n \to
\Wedge^{n_0} V_{n_0}, \text{ and} \\
m_F&:=m_{f_q} \circ \cdots \circ m_{f_{n+1}}: \Wedge^n V_n \to
\Wedge^q V_q
\end{align*}
and the $c_{e_i}$ and $m_{f_j}$ are as defined in \S\ref{ssec:Contraction}
and \S\ref{ssec:Multiplication}, respectively. Furthermore, since the
exterior powers are representations of the special orthogonal groups,
we may take $g,g',g''$ to be in $\SO(V_q), \SO(V_n), \SO(V_{n_0})$,
respectively.

We investigate the effect of the map on the left on (a pure tensor in
$\Wedge^n V_n$ corresponding to) a maximal (i.e., $n$-dimensional)
isotropic subspace
$W$ of $V_n$. First, $W$ is extended to $W':=W \oplus F$, then $g$ is
applied to $W'$, and the final contraction map sends $gW'$ to the image
in $V_{q}/E$ of $(gW') \cap E^{\perp}$.

Instead of intersecting $gW'$ with $E^\perp$, we may intersect
$W'=W \oplus F$ with $(E'')^\perp$ where $E'':=g^{-1} E$, followed by the
isometry $\ol{g}:(E'')^\perp
/ E'' \to E^\perp /E$ induced by $g$. Accordingly, one can verify that the map 
on the left-hand side of \eqref{eq:cgm} becomes (a scalar multiple of)
\[ \ol{g} \circ c_{E''} \circ m_F \]
where $c_{E''}:\Wedge^q V_q \to \Wedge^{n_0}((E'')^\perp / E'')$ is
the composition of contractions with a basis of $E''$, and where we
write $\ol{g}$ also for the map that $\ol{g}$ induces from
$\Wedge^{n_0} ((E'')^\perp / E'')$ to $\Wedge^{n_0} (E^\perp/E)$. 

Now consider the space $E'' \cap (V_n \oplus F) \subseteq V_q$. For $g$
in an open dense subset of $\SO(V_q)$, this intersection has the expected
dimension $(q-n_0) + (2n+q-n)-2q=n-n_0$, and for $g$ in an open dense
subset of $\SO(V_q)$ we also have $(E'')^\perp \cap F=\{0\}$ (because
$(E'')^\perp$ has codimension $q-n_0$, which is at least the dimension
$q-n$ of $F$). We restrict ourselves to such $g$. Then in particular $E''
\cap F=\{0\}$ and therefore the projection $\widetilde{E} \subseteq
V_n$ of $E'' \cap
(V_n \oplus F)$ along $F$ has dimension $n-n_0$, as well. Note that $\widetilde{E}$ is isotropic
because $E''$ is and because $F$ is the radical of the bilinear form
on $V_n \oplus F$. 

Furthermore, the projection $V_n \oplus F \to V_n$ restricts to a linear
isomorphism $(V_n \oplus F) \cap (E'')^\perp \to \widetilde{E}^\perp$,
where the latter is the orthogonal complement of $\widetilde{E}$ inside
$V_n$. This linear isomorphism induces an isometry
\[ h_1:((V_n \oplus F) \cap (E'')^\perp) / ((V_n \oplus F) \cap E'') \to 
\widetilde{E}^\perp / \widetilde{E} \]
between spaces of dimension $2n_0$ equipped with a nondegenerate
bilinear forms. On the other hand, the inclusion
$V_n \oplus F \to V_q$ also induces an isometry
\[ h_2:((V_n \oplus F) \cap (E'')^\perp) / ((V_n \oplus F) \cap E'')
\to (E'')^\perp / E''. \]
Now a computation shows that, up to a scalar, we have 
\[ c_{E''} \circ m_F = h_2 \circ h_1^{-1} \circ
c_{\widetilde{E}}, \]
where $c_{\widetilde{E}}:\Wedge^n V_n \to
\Wedge^{n_0}(\widetilde{E}^\perp/\widetilde{E})$ is
a composition of contractions with a basis of $\widetilde{E}$. Now
choose $g' \in \SO(V_n)$ such that $g' \widetilde{E}=E'$, so that we
have 
\[ c_{E'} \circ g' = \ol{g'} \circ c_{\widetilde{E}}, \]
where $\ol{g'}$ is the isometry
$\widetilde{E}^\perp/\widetilde{E} \to (E')^\perp / E'$ induced by
$g'$. We then conclude that 
\[ 
c_E \circ g \circ m_F = 
\ol{g} \circ h_2 \circ h_1^{-1} \circ (\ol{g'})^{-1} \circ c_{E'} \circ
g' \]
and hence we are done if we set 
\[ g'':=\ol{g} \circ h_2 \circ h_1^{-1} \circ (\ol{g'})^{-1}
\in \SO((E')^\perp / E')=\SO(V_{n_0}).  \qedhere
\]
\end{proof}

\end{document}